\pdfoutput=1
\documentclass[12pt]{article}
\usepackage[english]{babel}
\usepackage[utf8]{inputenc}
\usepackage[T1]{fontenc}
\usepackage{amsmath,amsthm,amsfonts,amscd,amssymb}
\usepackage{enumitem}
\usepackage{subcaption}
\usepackage{array}
\usepackage{tikz}

\usetikzlibrary{arrows}
\usetikzlibrary[patterns]
\usetikzlibrary{shapes.misc}
\usetikzlibrary{decorations.pathreplacing}

\tikzset{cross/.style={cross out, draw=black, minimum size=2*(#1-\pgflinewidth), inner sep=0pt, outer sep=0pt},
cross/.default={1pt}}

\usepackage[numeric,initials,nobysame,msc-links,abbrev]{amsrefs}
\renewcommand{\eprint}[1]{\href{https://arxiv.org/abs/#1}{arXiv:#1}}
\newcommand{\pageafter}[1]{#1~pp.}
\BibSpec{article}{%
+{} {\PrintAuthors} {author}
+{,} { \textit} {title}
+{.} { } {part}
+{:} { \textit} {subtitle}
+{,} { \PrintContributions} {contribution}
+{.} { \PrintPartials} {partial}
+{,} { } {journal}
+{} { \textbf} {volume}
+{} { \PrintDatePV} {date}
+{,} { \issuetext} {number}
+{,} { \pageafter} {pages}
+{,} { } {status}
+{,} { \PrintDOI} {doi}
+{,} { available at \eprint} {eprint}
+{} { \parenthesize} {language}
+{} { \PrintTranslation} {translation}
+{;} { \PrintReprint} {reprint}
+{.} { } {note}
+{.} {} {transition}
+{} {\SentenceSpace \PrintReviews} {review}
}
\BibSpec{collection.article}{%
+{} {\PrintAuthors} {author}
+{,} { \textit} {title}
+{.} { } {part}
+{:} { \textit} {subtitle}
+{,} { \PrintContributions} {contribution}
+{,} { \PrintConference} {conference}
+{} {\PrintBook} {book}
+{,} { } {booktitle}
+{,} { \PrintDateB} {date}
+{,} { \pageafter} {pages}
+{,} { } {status}
+{,} { \PrintDOI} {doi}
+{,} { available at \eprint} {eprint}
+{} { \parenthesize} {language}
+{} { \PrintTranslation} {translation}
+{;} { \PrintReprint} {reprint}
+{.} { } {note}
+{.} {} {transition}
+{} {\SentenceSpace \PrintReviews} {review}
}
\usepackage{geometry}
\usepackage{authblk}
\usepackage{bbm}
\usepackage{hyperref}
\usepackage[capitalise]{cleveref}
\usepackage{xparse}

\setlist[itemize]{leftmargin=*}
\setlist[enumerate]{leftmargin=*,label=(\arabic*),ref=(\arabic*)}

\newtheorem{theorem}{Theorem}
\crefname{theorem}{Theorem}{Theorems}
\newtheorem{corollary}[theorem]{Corollary}
\crefname{corollary}{Corollary}{Corollaries}
\newtheorem{lemma}[theorem]{Lemma}
\crefname{lemma}{Lemma}{Lemmas}
\newtheorem{proposition}[theorem]{Proposition}
\crefname{proposition}{Proposition}{Propositions}
\newtheorem{conjecture}[theorem]{Conjecture}
\crefname{conjecture}{Conjecture}{Conjectures}

\crefname{question}{Question}{Questions}
\theoremstyle{definition}
\newtheorem{definition}[theorem]{Definition}
\crefname{definition}{Definition}{Definitions}

\crefname{remark}{Remark}{Remarks}

\crefname{example}{Example}{Examples}
\newtheorem{observation}[theorem]{Observation}
\crefname{observation}{Observation}{Observations}
\newtheorem{claim}[theorem]{Claim}
\crefname{claim}{Claim}{Claims}

\crefname{assumption}{Assumption}{Assumptions}

\numberwithin{theorem}{section}
\numberwithin{equation}{section}

\newcommand{\Z}{\mathbb{Z}}
\newcommand{\N}{\mathbb{N}}

\renewcommand{\P}{\mathbb{P}}

\newcommand{\I}{\mathcal{I}}
\newcommand{\A}{\mathcal{A}}
\newcommand{\E}{\mathcal{E}}
\newcommand{\e}{\varepsilon}
\newcommand{\HH}{\mathcal{H}}

\newcommand{\cA}{\ensuremath{\mathcal A}}

\newcommand{\cC}{\ensuremath{\mathcal C}}
\newcommand{\cD}{\ensuremath{\mathcal D}}
\newcommand{\cE}{\ensuremath{\mathcal E}}
\newcommand{\cF}{\ensuremath{\mathcal F}}

\newcommand{\cH}{\ensuremath{\mathcal H}}
\newcommand{\cI}{\ensuremath{\mathcal I}}

\newcommand{\cK}{\ensuremath{\mathcal K}}

\newcommand{\cS}{\ensuremath{\mathcal S}}
\newcommand{\cT}{\ensuremath{\mathcal T}}
\newcommand{\cU}{\ensuremath{\mathcal U}}

\newcommand{\cW}{\ensuremath{\mathcal W}}

\newcommand{\bbE}{{\ensuremath{\mathbb E}} }

\newcommand{\bbH}{{\ensuremath{\mathbb H}} }

\newcommand{\bbN}{{\ensuremath{\mathbb N}} }

\newcommand{\bbP}{{\ensuremath{\mathbb P}} }

\newcommand{\bbR}{{\ensuremath{\mathbb R}} }

\newcommand{\bbZ}{{\ensuremath{\mathbb Z}} }

\newcommand{\ba}{\ensuremath{\mathbf{a} }}
\newcommand{\bb}{\ensuremath{\mathbf{b} }}
\newcommand{\bc}{\ensuremath{\mathbf{c} }}
\newcommand{\bd}{\ensuremath{\mathbf{d} }}
\newcommand{\be}{\ensuremath{\mathbf{e} }}

\newcommand{\bm}{\ensuremath{\mathbf{m} }}

\newcommand{\bs}{\ensuremath{\mathbf{s} }}

\renewcommand{\>}{\rangle}

\newcommand{\1}{{\ensuremath{\mathbbm{1}}} }

\DeclareDocumentCommand \to { o o } {%
  \IfNoValueTF {#1} {\IfNoValueTF{#2}{\rightarrow}{\xrightarrow[#2]}}%
{\IfNoValueTF{#2}{\xrightarrow{#1}}{\xrightarrow[#2]{#1}}}}
\DeclareDocumentCommand \nto { o } {%
  \IfNoValueTF {#1} {{\ensuremath{\centernot{\rightarrow}}} }{{\ensuremath{\centernot{\xrightarrow{#1}}}}}%
}

\DeclareMathOperator{\diam}{diam}
\DeclareMathOperator{\Pod}{Pod}
 
\title{Sharp metastability transition for two-dimensional bootstrap percolation with symmetric isotropic threshold rules}
\date{\today}
\author[1,2]{Hugo Duminil-Copin}
\author[3]{Ivailo Hartarsky}
\affil[1]{Universit\'e de Gen\`eve, Section de Math\'ematiques, 2-4 rue du Li\`evre 1211 Geneva, Switzerland}
\affil[2]{Institut des Hautes \'Etudes Scientifiques, 35 route de Chartres
91440 Bures-sur-Yvette, France}
\affil[3]{TU Wien, Faculty of Mathematics and Geoinformation, Institute of Statistics and Mathematical Methods in Economics, Research Unit of Probability, Wiedner Hauptstra\ss e 8-10, A-1040 Vienna, Austria, \texttt{ivailo.hartarsky@tuwien.ac.at}}
\begin{document}

\maketitle

\begin{abstract} 
We study two-dimensional critical bootstrap percolation models. We establish that a class of these models including all isotropic threshold rules with a convex symmetric neighbourhood, undergoes a sharp metastability transition. This extends previous instances proved for several specific rules. The paper supersedes a draft by Alexander Holroyd and the first author from 2012. While it served a role in the subsequent development of bootstrap percolation universality, we have chosen to adopt a more contemporary viewpoint in its present form.
\end{abstract}


\section{Introduction}
A threshold bootstrap percolation model is a simple cellular automaton that provides a useful model for studying several phenomena such as metastability, dynamics of glasses or crack formation. A famous example of a threshold model is the \emph{$2$-neighbour bootstrap percolation} originally introduced by Chalupa, Leath and Reich \cite{Chalupa79} (also see \cite{Kogut81}). In this model, sites of the square lattice $\mathbb Z^2$ are infected or healthy. At time 0, sites are infected with probability $p$ independently of each other (we denote the corresponding measure by $\mathbb P_p$). At each time step, a site becomes infected if two or more of its nearest neighbours are infected.

The first rigorous result on this model \cite{vanEnter87}, dating back to 1987, established that every site of $\mathbb Z^2$ becomes infected almost surely whenever $p>0$. This motivates the study of the (random) first time $\tau$ at which 0 becomes infected as $p$ goes to 0. In \cite{Aizenman88}, Aizenman and Lebowitz proved that there exist two constants $c,C\in(0,\infty)$ such that
\[\lim_{p\rightarrow 0}\mathbb P_p\left(e^{c/p}\le \tau\le e^{C/p}\right)=1.\] 
We refer to this article for an enlightening exposition of the metastability effects in the model. The question of whether $c$ and $C$ could be chosen arbitrary close to each other was left open for a long time. Finally, a sharp metastability transition was shown to occur in \cite{Holroyd03}: $p\log \tau$ converges in probability to $\pi^2/18$ as  $p\rightarrow 0$. More precise estimates for $\tau$ were derived later in \cite{Hartarsky24locality}.

Several authors investigated more general growth rules and the right order of magnitude for $\log\tau$ is now known for all rules \cite{Bollobas23}, hence generalising the result of Aizenman and Lebowitz. The sharp metastability transition, though, remained available only for a handful of isolated examples \cites{Bollobas17,Duminil-Copin13,Holroyd03,Holroyd03a}. The goal of this paper is to prove sharp metastability for a wide class of models. In particular, we show that every isotropic symmetric convex threshold bootstrap percolation model exhibits a sharp transition.

\subsection{\texorpdfstring{$\cU$}{U}-bootstrap percolation}
\label{subsec:BP}

Let $\Z^2=\{x=(x_1,x_2):x_1,x_2\in \Z\}$ be the set of all 2-vectors of integers and $\mathbb N=\{0,1,\dots\}$. Elements of $\Z^2$ are called {\em sites}. An \emph{update rule} is any finite non-empty subset of $\bbZ^2\setminus\{0\}$. An \emph{update family} is a finite non-empty set of update rules. An update family $\cU$ is \emph{symmetric}, if for every $U\in\cU$ we have $-U=\{-x:x\in U\}\in\cU$. Given an update family $\cU$ and a set $A=A_0\subseteq\bbZ^2$ of \emph{initial infections}, we recursively define
\[A_{t+1}=A_t\cup\{x\in\bbZ^2:\exists U\in\cU,\forall u\in U,x+u\in A_t\}\]
to be the set of sites \emph{infected at time $t$} in the $\cU$-bootstrap percolation process. The set $[A]=\bigcup_{t\ge 0}A_t$ of eventually infected sites is called the \emph{closure} of $A$. A set $A\subseteq\bbZ^2$ is called \emph{stable} if $[A]=A$. An observable of particular interest is the infection time of the origin \[\tau=\inf\{t\in\bbN:0\in A_t\}\in\bbN\cup\{\infty\}.\] We will systematically be interested in the asymptotics of $\tau$ when each site is initially infected independently with probability $p\to 0$. We denote the corresponding distribution of $A$ by $\bbP_p$.

Among all update families, threshold rules initially received particular attention \cite{Gravner99}. They are defined by a finite \emph{neighbourhood} $\cK\subset \bbZ^2$, containing $0$, and a positive integer \emph{threshold} $\theta$. Then \[\cU(\cK,\theta)=\{U\subseteq \cK\setminus\{0\}:|U|=\theta\}\] is the associated update family. In other words, a site $x$ becomes infected if at least $\theta$ of the sites in its neighbourhood $x+\cK$ are already infected. A set $\cK\subseteq \bbR^2$ is called \emph{symmetric} if $x\in \cK$ implies $-x\in \cK$ for all $x\in\bbR^2$. We say that a neighbourhood $\cK\subset \bbZ^2$ is \emph{convex symmetric}, if it is the intersection of a bounded convex symmetric subset of $\bbR^2$ with $\bbZ^2$. We say that a neighbourhood $\cK\subset\bbZ^2$ is \emph{two-dimensional} if $\cK\not\subset u\bbR$ for every $u\in\bbR^2$. An example of a two-dimensional convex symmetric neighbourhood is given in \cref{fig:example} and can also serve as illustration for the definitions to follow.

We will require a few definitions from the bootstrap percolation universality framework \cites{Bollobas15,Bollobas23,Gravner99,Hartarsky22univlower}. A \emph{direction} is a unit vector of $\bbR^2$, viewed as an element of the unit circle $S^1$. We denote the open half plane with outer normal $u$ by $\bbH_u=\{x\in\bbZ^2:\<u,x\><0\}$ and its boundary by  $l_u=\{x\in\bbZ^2:\<u,x\>=0\}$. A direction $u\in S^1$ is called \emph{stable}, if $\bbH_u$ is stable. The direction is \emph{unstable} otherwise, which can be reinterpreted as follows: there exists an update rule $U\subset\bbH_u$. In the case of a threshold rule, unstable directions $u$ are those for which $|\bbH_u\cap \cK|\ge \theta$.

A direction $u\in S^1$ is called \emph{rational} if $\lambda u\in \bbZ^2\setminus\{0\}$ for some $\lambda\in\bbR$. In this case, we denote $\rho_u=\min\{\rho>0:\exists x\in\bbZ^2,\<u,x\>=\rho\}$. Then $u\bbR\cap\bbZ^2=(u/\rho_u)\bbZ$. Thus, it will be convenient to define $u^\perp=(u_2,-u_1)/\rho_u$, so that $l_u=u^\perp\bbZ$. We further denote by  $l_u(n)=\{x\in\bbZ^2:\<u,x\>=n\rho_u\}$ the $n$-th line perpendicular to $u$, so that $\bbZ^2=\bigsqcup_{n\in\bbZ}l_u(n)$. Note that for any $n\in\bbZ$, $l_u(n)$ is a translate of $l_u$.

For an isolated stable or an unstable direction $u\in S^1$, we define its \emph{difficulty}
\begin{equation}
\label{eq:def:alpha}
\alpha(u)=\min\{|Z|:Z\subset\bbZ^2,|[\bbH_u\cup Z]\setminus\bbH_u|=\infty\}\in\bbN.
\end{equation}
That is, the difficulty of $u$ is the minimal number of infected sites needed in addition to the half-plane $\bbH_u$, so that infinitely many additional sites become infected. An update family is called \emph{isotropic} if it has a finite but nonzero number of stable directions and each open semicircle of $S^1$ contains a stable direction of maximal difficulty.\footnote{The reader who is familiar with the jargon of \cite{Bollobas23} can check that a critical symmetric update family is isotropic if and only if it is balanced.} For isotropic models we call 
\begin{equation}
\label{eq:def:alpha:global}\alpha=\max_{u\in S^1}\alpha(u)
\end{equation} the \emph{difficulty} of the update family. A set $Z$ realising the minimum in \cref{eq:def:alpha} is called a \emph{helping set}. A helping set $Z\subset \bbZ^2$ for $u$ is \emph{voracious} if $[\bbH_u\cup Z]\supseteq l_u$. An isolated stable direction is called \emph{voracious}, if every helping set for it is voracious. The update family is called \emph{voracious} if all isolated stable directions are voracious.

It was shown in \cite{Bollobas23} that for every isotropic update family, there exists $C>c>0$ such that 
\[\lim_{p\rightarrow 0}\mathbb P_p\left(e^{ c/p^{\alpha}}<\tau<e^{ C/p^{\alpha}}\right)=1.\]  
One can check that symmetric threshold models are isotropic if and only if the maximum $\iota(\cK)=\max_{u\in S^1}|l_u\cap \cK|$ is attained for at least two non-opposite directions $u$ and $|\cK|-\iota(\cK)<2\theta<|\cK|$. In that case, the difficulty is given by $\alpha=\theta-(|\cK|-\iota(\cK))/2$ and the difficulty of a direction $u\in S^1$ is $\alpha(u)=\max(0,\theta-|\cK\setminus l_u|/2)$ (see \cite{Gravner99}).

\subsection{Main results}
\label{subsec:results}
Our main result is the following.
\begin{theorem}\label{th:main}
For any symmetric voracious isotropic update family with difficulty $\alpha$, there exists $\lambda\in(0,\infty)$ such that for all $\varepsilon>0$,
\[\lim_{p\rightarrow0}\mathbb P_p\left(|p^\alpha\log \tau-\lambda|>\varepsilon\right)=0.\]
\end{theorem}
The constant $\lambda$ is identified as the solution of a variational problem, see \cref{def:W}. Roughly speaking, $\lambda$ quantifies the probability of the optimal way for a small polygonal region of infections, whose directions are dictated by the stable directions, to expand to infinity. While our definition of $\lambda$ is quite implicit, we expect that the method of \cite{Hartarsky24locality} can be used to provide numerical estimates of this constant. We note that symmetry will only be used in \cref{sec:lower}, where the lower bound on $\tau$ is proved, but not for the upper one.

We further show that voracity is rather common.
\begin{proposition}\label{prop:voracity}
For every threshold rule with two-dimensional convex symmetric neighbourhood, all isolated stable directions are voracious.
\end{proposition}
This proposition provides a partial answer to Gravner and Griffeath \cite{Gravner99}, who wrote ``Such examples suggest the possibility of a general theory to the effect that voracity should be automatic for `nice' $\cK$ . Such a theory is far from easy to develop; at present the only result in
this direction, due to Bohman \cite{Bohman99a} and proved by a complex combinatorial argument, applies to supercritical threshold growth dynamics on box neighborhoods.'' We should note that, given an arbitrary update family, using the algorithm of \cite{Hartarsky20a}, one can also determine whether it is voracious. However, as shown in \cite{Hartarsky20a}, determining the difficulty $\alpha$ is NP-hard, so we expect that determining whether an update family is voracious is also NP-hard.

In addition to convex symmetric threshold rules treated by \cref{prop:voracity}, to the best of our knowledge, all commonly studied isotropic update families are both symmetric and voracious---the $k$-cross model, Frob\"ose bootstrap percolation, modified and non-modified 2-neighbour bootstrap percolation (on $\bbZ^2$), 3-neighbour bootstrap percolation on the triangular lattice. However, for these last examples \cref{th:main} and more is already known \cites{Bringmann12,Hartarsky24locality}. Therefore, the importance of our result stems from its universality.

It is known that beyond the class of isotropic models, asymptotic behaviours that differ from the one in \cref{th:main} are displayed \cites{Bollobas23,Balister16,Bollobas15}. Nevertheless, the result should hold in yet greater generality, as discussed in \cref{sec:future}. On the other hand, it should be noted that our techniques in conjunction with those of \cites{Hartarsky24univupper,Hartarsky23FA} should lead to sharp threshold results like \cref{th:main} with $\lambda$ replaced by $2\lambda$ for symmetric voracious isotropic kinetically constrained models.

\subsection{Organization of the paper}
The rest of the paper is organised as follows. We begin by proving the combinatorial result of \cref{prop:voracity} in \cref{sec:convex}. We provide the setup for the proof of \cref{th:main} in \cref{sec:setup}. In particular, we introduce the notions of traversability and droplets, and define the constant $\lambda$ appearing in \cref{th:main}. \Cref{sec:upper} proves the upper bound of \cref{th:main}, while \cref{sec:lower} proves the lower one. Finally, in \cref{sec:future}, we discuss possible future directions and generalisations of \cref{th:main}.

\section{Convex symmetric threshold rules}
\label{sec:convex}
In this section, we establish \cref{prop:voracity} in order to better familiarise ourselves with helping sets. For the rest of the section, we fix a two-dimensional convex symmetric neighbourhood $\cK\ni 0$ and threshold $\theta$. We further fix an isolated stable direction $u$. Thus, 
\begin{equation}
\label{eq:alpha:K:bound}
\alpha(u)=\theta-|\cK\setminus l_u|/2=\theta-|\cK\cap\bbH_u|>0.
\end{equation} 
Since $\cK\cap l_u\neq\varnothing$, $u$ is necessarily rational. The next preparatory result states that $\cK$ intersects the first line with normal $u$ and, if $\cK$ reaches beyond the first line, then it contains at least $\alpha(u)$ sites on the first line. See \cref{fig:example} for an illustration.
\begin{lemma}
\label{lem:convex:first:line}
We have $l_u(1)\cap \cK\neq\varnothing$. Moreover, if $l_u(2)\cap \cK\neq \varnothing$, then $|l_u(1)\cap \cK|\ge \alpha(u)$.
\end{lemma}
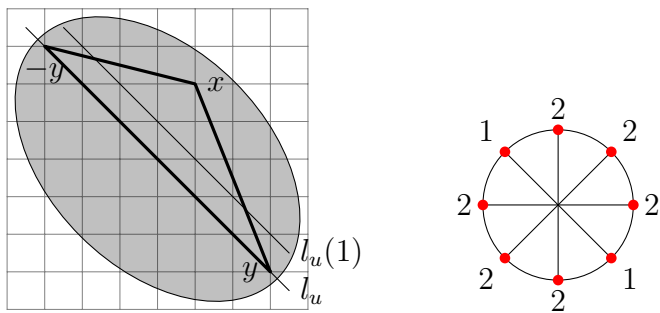
\begin{figure}
    \centering
    \begin{tikzpicture}[line cap=round,line join=round,>=triangle 45,x=0.5cm,y=0.5cm]
\draw[step=1.0,gray,very thin] (-4,-4) grid (4,4);
\draw [rotate around={-45:(0,0)},fill=black,fill opacity=0.25] (0,0) ellipse (4.5 and 2.9);
\draw [domain=-3.5:3.5] plot(\x,{(-0--1*\x)/-1});
\draw [very thick] (-3,3) node[below]{$-y$}-- (1,2) node[right]{$x$}-- (3,-3) node[left]{$y$}--cycle;
\draw [domain=-2.5:3.5] plot(\x,{(-1--1*\x)/-1});
\draw (3.5,-3.5) node[right]{$l_u$};
\draw (3.5,-2.5) node[right]{$l_u(1)$};
\end{tikzpicture}
\qquad
\begin{tikzpicture}
\draw(0,0) circle (1);
\draw (0,0)-- (1,0);
\draw (0,1)-- (0,0);
\draw (0,0)-- (-1,0);
\draw (0,0)-- (0,-1);
\draw (-0.707,-0.707)-- (0.707,0.707);
\draw (-0.707,0.707)-- (0.707,-0.707);
\fill [color=red] (0,1) circle (2pt) node[anchor=south,black] {$2$};
\fill [color=red] (0,-1) circle (2pt) node[anchor=north,black] {$2$};
\fill [color=red] (1,0) circle (2pt) node[anchor=west,black] {$2$};
\fill [color=red] (-1,0) circle (2pt) node[anchor=east,black] {$2$};
\fill [color=red] (0.707,0.707) circle (2pt) node[anchor=south west,black] {$2$};
\fill [color=red] (-0.707,-0.707) circle (2pt) node[anchor=north east,black] {$2$};
\fill [color=red] (0.707,-0.707) circle (2pt) node[anchor=north west,black] {$1$};
\fill [color=red] (-0.707,0.707) circle (2pt) node[anchor=south east,black] {$1$};
\end{tikzpicture}
    \caption{Example of an isotropic threshold rule. The neighbourhood $\cK$ is the intersection of $\bbZ^2$ with the shaded convex symmetric subset of $\bbR^2$. The threshold $\theta$ can be taken equal to $18$, making the model isotropic with $\alpha=2$ and $\alpha(u)=2$. It has six stable directions with difficulty $3$ and two with difficulty $1$, as depicted on the right. The triangle $T$ used in the proof of \cref{lem:convex:first:line} is thickened.}
    \label{fig:example}
\end{figure}
\begin{proof}
Since $\cK$ is two-dimensional, we have $\cK\not\subset l_u$. If $l_u\not\supset\cK\subset l_u\cup l_u(-1)\cup l_u(1)$, we are done by symmetry. Let $x\in \cK\cap l_u(n)$ for some $n\ge 2$. Since $u$ is stable, $|\cK\cap l_u|\ge 3$ by symmetry. Let $y=u^\perp(|\cK\cap l_u|-1)/2\in \cK\cap l_u$, that is, the last site of $\cK$ on $l_u$. Consider the  
triangle $T\subset \bbR^2$ with vertices $x$, $y$, $-y$. By convexity the lattice sites in $T$ are in $\cK$. But the length of the $(-y,y)$ side of $T$ is $2\|y\|=(|\cK\cap l_u|-1)/\rho_u$ and the height from $x$ has length $\<x,u\>\ge 2\rho_u$. Therefore, by Thales theorem, the segment
\[\{t\in T:\<t,u\>=\rho_u\}\]
has length at least $(|\cK\cap l_u|-1)/(2\rho_u)\ge \alpha(u)/\rho_u$. Since $l_u(1)=\{z\in\bbZ^2:\<z,u\>=\rho_u\}$ is a translate of $u^\perp\bbZ$, it necessarily intersects this segment in at least $\alpha(u)$ points.
\end{proof}
We are now ready to prove \cref{prop:voracity}, but before spelling out the details, let us give an outline of the argument. We consider an isolated stable direction $u$ and initial condition given by the half-plane $\bbH_u$ and a helping set $H$. We observe that on the first step infections only appear in $l_u$. If infections never appear outside $l_u$, we consider a site sufficiently far from $H$ on $l_u$, which becomes infected only using infections in one direction on $l_u$, and propagate infection sequentially from it.

If infections appear outside $l_u$, we first assume $\cK$ is contained in three lines and so is $H$. Then infection is shown to spread sequentially starting from the first site to become infected. If, on the contrary, $\cK$ is more spread out, we use \cref{lem:convex:first:line} to show that a large segment on $l_u$ needs to become infected in order to infect anyone beyond the first line. This segment is sufficient to completely infect $l_u$.
\begin{proof}[Proof of \cref{prop:voracity}]
Let $H$ be a helping set for $u$. That is, a set with $|H|=\alpha(u)=\theta-|\cK\setminus l_u|/2$ such that $|[H\cup\bbH_u]\setminus\bbH_u|=\infty$. By \cref{lem:convex:first:line}, we know that on the first step of the bootstrap percolation dynamics with initial condition $H\cup\bbH_u$, only sites in $l_u$ become infected. Indeed, for $x\in l_u(n)$ with $n\ge 1$ we have \[(x+\cK)\cap (H\cup\bbH_u)\le |H|+|\cK\cap\bbH_u\setminus l_u(-1)|<\theta-|\cK\setminus l_u|/2+|\cK\cap \bbH_u|=\theta.\]

Assume that $[H\cup\bbH_u]\setminus (H\cup \bbH_u)\subseteq l_u$. By symmetry, without loss of generality we may consider a site $y\in l_u\cap[H\cup \bbH_u]$ such that $\<y,u^\perp\>>\max\<h+k,u^\perp\>$ for all $h\in H$ and $k\in \cK$. Further choose $y$ such that no site $z\in l_u$ with $\<z,u^\perp\>>\<y,u^\perp\>$ is infected before $y$. Then there are at least $\alpha(u)$ infected sites in $y+\cK\cap l_u$ before $y$ becomes infected. But then on the next step there are also at least $\alpha(u)$ infected sites in $y+u^\perp+\cK\cap l_u$ (including $y$). Proceeding by induction, we see that for any $m\in\bbZ$ the site $y+mu^\perp$ becomes infected at most $|m|$ steps after $y$, which concludes the proof of the voracity of $u$.

Assume, on the contrary, that some site outside $l_u$ becomes infected. This entails $H\cap l_u=\varnothing$ since otherwise there are at most $\alpha(u)-1<\theta-|\cK\cap\bbH_u|$ sites outside $\bbH_u\cup l_u$. We consider two cases. 

Firstly, assume that $\cK\subset l_u\cup l_u(-1)\cup l_u(1)$ and let $x\in l_u$ be a site infected on the first step. As in the calculation above we need to have $H\subseteq (x+\cK)\setminus \bbH_u$, so $H\subset l_u(1)$. We claim that $x+u^\perp$ becomes infected on the second step or earlier. Indeed, $x\in x+u^\perp+\cK$ and $\cK\cap l_u(1)$ is a discrete interval, so $|(x+u^\perp+\cK)\cap H|\ge |(x+\cK)\cap H|-1=\alpha(u)-1=\theta-|\cK\cap\bbH_u|$. Reasoning similarly by induction, we see that all sites in $(x+\cK)\cap l_u$ become infected. However, they are enough to infect $l_u$ on their own, since the first site in $y\in l_u$ outside $x+\cK$ has at least $(|\cK\cap l_u|-1)/2\ge \alpha(u)$ sites in $(x+\cK)\cap (y+\cK)$, which we already established to be infected.

Secondly, assume that $\cK\cap l_u(n)\neq\varnothing$ for some $n\ge 2$. Observe that by \cref{lem:convex:first:line} this implies that $|\cK\cap l_u(1)|\ge (|\cK\cap l_u|-1)/2\ge \alpha(u)$. Consider the first site $x\not\in l_u$ which becomes infected and let $m\ge1$ be such that $x\in l_u(m)$. Then the number of infected sites in $x+\cK$ just before $x$ is infected is at most $|H|+|\cK\cap\bbH_u|=\theta$. In order to infect $x$ we need to have equality, so all sites in $(x+\cK)\cap l_u(m-1)$ are infected before $x$. By our choice of $x$ this means that $m=1$ and there are at least $\alpha(u)$ consecutive sites infected in $l_u$. As above, this is enough to infect all of $l_u$, concluding the proof.
\end{proof}

\section{Setup}
\label{sec:setup}
We fix an update family $\cU$ for the rest of the paper.
\subsection{Probabilistic tools}
\label{subsec:proba:tools}
An event $E\subseteq\Omega=\{A:A\subset\bbZ^2\}$ is \emph{increasing} if $A\in E$ and $A\subseteq A'$ imply $A'\in E$. Two important correlation inequalities related to increasing events will be used in the article.

The first one is the \emph{Harris inequality} \cite{Harris60} stating that for two increasing events $E,F$,
\begin{equation}
\label{eq:Harris}\P_p(E\cap F)\geq \P_p(E)\P_p(F).\end{equation}
The second one is the \emph{BK inequality} \cite{BK85}. For $E$ and $F$ two increasing events, their \emph{disjoint occurrence} $E\circ F$ is defined as follows. A configuration $A\in\Omega$ belongs to $E\circ F$ if there exists a set $B\subseteq A$ such that $B\in E$ and $A\setminus B\in F$. For $k$ increasing events $E_1,\dots,E_k$, one can define the disjoint occurrence by
\[E_1\circ\cdots\circ E_k=E_1\circ(E_2\circ\cdots(E_{k-1}\circ E_k)),\]
that is, $E_1,\dots,E_k$ admit disjoint witness sets. Then, for any increasing events $E_1,\dots,E_k$ depending on a finite number of sites, the BK inequality reads
\begin{equation}
\label{eq:BK}\P_p(E_1\circ\cdots\circ E_k)\leq \P_p(E_1)\cdots\P_p(E_k).
\end{equation}
We refer the reader to the book \cite{Grimmett99} for proofs of these two classical inequalities.

\subsection{The traversability functions \texorpdfstring{$h^u$}{hu}}
\label{subsec:hu}
The following definition is an extension of the definition of occupied rows and columns for the simple bootstrap percolation, see \cite{Holroyd03}. In words, an occupied line has a helping set sufficient for infection to invade it in direction $u$.
\begin{definition}[Occupied lines]
\label{def:occupied}
For an isolated stable direction $u$, let $\cH^u$ denote the set of helping sets for $u$ (recall \cref{subsec:BP}). A line $l_u(n)$ orthogonal to $u$ is \emph{occupied} in $A\subseteq\bbZ^2$ if there exist $x\in l_u(n)$ and $H\in \cH^u$ such that $x+H\subseteq A$. 
\end{definition}

We call a \emph{rectangle} any translate of the set
\[R^u(m,n)=\big\{x\in \Z^2:0\le \langle x,u^\perp\rangle <  m/\rho_u^2\text{ and }0\le\langle x,u\rangle < n\rho_u\big\}\]
for some $m,n\in\bbN$. With this notation we have that $R^u(m,n)\cap l_u$ contains $m$ sites and $R^u(m,n)\subset\bigcup_{j=0}^{n-1}l_u(j)$, so the rectangle spans $n$ lines. We naturally define $R^u(m,\infty)=\bigcup_{n> 0} R^u(m,n)$. Define the event
\begin{equation}
\label{eq:def:Aumn}
\A^u(m,n)=\bigcap_{j=0}^{n-1}\big\{l_u(j)\text{ is occupied in }A\cap R^u(m,\infty)\big\}.\end{equation}
Note that this event depends on the state of sites in the slightly higher rectangle \begin{equation}
\label{eq:taller:rectangle}R^u(m,n+\max\{\<x,u\>/\rho_u:x\in H,H\in\cH^u\}).\end{equation}

The following proposition studies the behaviour of $\mathbb P_p[\A^u(m,n)]$. In particular, we prove that this probability can be expressed in terms of a family of functions $h^u_p$. They are obtained using various sub-additivity properties corresponding to cutting rectangles into smaller pieces.

\begin{proposition}\label{prop:h}Let $u$ be an isolated stable direction. There exists a constant $V_u>0$ and a family of continuous non-increasing functions $(h_p^u)_{p\in(0,1)}:(0,\infty)\to(0,\infty)$ such that
\begin{enumerate}
\item\label{item:1} \textup{(Link to $\cA^u$)} For any $p\in(0,1)$, $m>V_u$ and $n>0$,
\begin{equation}\label{proposition h}\exp\left(-h_p^u\left(p^{\alpha(u)}m\right)(n+V_u)\right)\leq \P_p(\A^u(m,n)) \leq \exp\left(-h_p^u\left( p^{\alpha(u)}m\right)(n-V_u)\right).\end{equation}

\item\label{item:2} \textup{(Behaviour near 0 and $\infty$)} There exist $p_0,c_{u}>0$ such that for every $p\le p_0$ and  $x\ge p^{\alpha(u)}/c_{u}$, 
\begin{equation}
-c_{u}\log\left(1-e^{-x/c_{u}}\right)\leq h_p^u(x)\leq -\log\left(1-e^{-c_{u}x}\right)\label{estimate}.
\end{equation}
\item\label{item:3} \textup{(Uniform convergence)} There exists a continuous non-increasing integrable function $h^u:\mathbb (0,\infty)\rightarrow(0,\infty)$ such that,  as $p\to 0$, $h^u_p/h^u$ converges to 1 uniformly on $(a,b)$ for every $a,b>0$.
\end{enumerate}
\end{proposition}

In simple cases, the functions $h^u$ could be computed explicitly. The limit $h^u$ corresponds to the functions $f$ and $g$ in \cite{Holroyd03} and functions $g_k$ in \cite{Holroyd03a}. However, in general, these functions are not explicit. Also note that if $m$ and $n$ are of order $p^{-\alpha(u)}$, then $-p^{\alpha(u)}\log \P_p(\mathcal A^u(m,n))$ remains of order 1 when $p$ goes to 0. This is why $p^{-\alpha(u)}$ is the right scale to consider.

\begin{proof}[Proof of \ref{item:1}] The main ingredient to construct $h^u_p$ is sub- and super-multiplicativity. Fix $m$ large enough so that $\bbP_p(\cA^u(m,n))\in(0,1)$ for all $n>0$ and $p\in(0,1)$. That is, $m$ is chosen so that the rectangle $R^u(m,\infty)$ is wide enough to fit helping sets for any line. Define $v_{p,m}(n)=\P_p(\A^u(m,n))$. 

Observe that $R^u(m,n+n')=R^u(m,n)\sqcup (n\rho_u u+R^u(m,n'))$ for any $n'\in\bbN$ and $n\in \rho_u^{-2}\bbN$, so $v_{p,m}(n)v_{p,m}(n')\le v_{p,m}(n+n')$ by  the Harris inequality \eqref{eq:Harris}. Similarly, recalling \cref{eq:taller:rectangle}, we can choose an integer constant $C>0$ divisible by $1/\rho_u^2$ such that for $n\in \rho_u^{-2}\bbN$ with $n\ge C$ and $n'\in\bbN$, we have $v_{p,m}(n+n')\le v_{p,m}(n-C)v_{p,m}(n')$ by independence. Thus for any $n,n'\in1/\rho_u^2\bbN$,
\[v_{p,m}(n)v_{p,m}(n')\le v_{p,m}(n+n')\le v_{p,m}(n-C)v_{p,m}(n').\]
The sub-additivity lemma and the first inequality imply that there exists $\mu=\mu(u,p,m)\in(0,1]$ such that
$v_{p,m}(n)\le \mu^{n}$ for every $n\in1/\rho_u^2\bbN$ and \begin{equation}
\label{eq:def:mu}
\lim_{n\to\infty}(v_{p,m}(n/\rho_u^2))^{\rho_u^2/n}=\mu.\end{equation}
Moreover, the second inequality entails that for any $k\in\bbN$, $n\in1/\rho_u^2\bbN$,
\[\frac{v_{p,m}((k+1)(n+C))}{\mu^{(k+1)(n+C)}}\le \frac{v_{p,m}(k(n+C))}{\mu^{k(n+C)}}\frac{v_{p,m}(n)}{\mu^{n+C}}\le \left(\frac{v_{p,m}(n)}{\mu^{n+C}}\right)^k,\]
so $v_{p,m}(n)\ge\mu^{n+C}$. Finally, recalling that $v_{p,m}$ is non-increasing by definition, we obtain that for any $n\in\bbN$
\begin{equation}
\label{eq:vpm:bounds}\mu^{C+\lceil \rho_u^2n\rceil/\rho_u^2}\le v_{p,m}(n)\le \mu^{\lfloor \rho_u^2n\rfloor/\rho_u^2}.
\end{equation}
Since $p\neq 1$, this is clearly implies $\mu\neq1$.
For any $m\in \mathbb N$, set $h^u_p(p^{\alpha(u)}m)=-\log \mu$. Extend $h^u_m$ to all $(0,\infty)$ in a piece-wise linear way. Note that $h^u_p$ is non-increasing since $\A^u(m,n)\subseteq \A^u(m+1,n)$ for every $n,m\in\bbN$.
\end{proof}
\begin{figure}
    \centering
\begin{tikzpicture}[line cap=round,line join=round,>=triangle 45,x=0.67cm,y=1cm]
\clip(0,0) rectangle (21,3.2);
\draw (3,0)-- (0,0);
\draw (0,0)-- (0,1);
\draw (0,1)-- (3,1);
\draw (3,1)-- (3,0);
\draw (18,0)-- (15,0);
\draw (15,0)-- (15,1);
\draw (15,1)-- (18,1);
\draw (18,1)-- (18,0);
\draw (6,0.2)-- (3,0.2);
\draw (3,0.2)-- (3,1.2);
\draw (3,1.2)-- (6,1.2);
\draw (6,1.2)-- (6,0.2);
\draw (21,0.2)-- (18,0.2);
\draw (18,0.2)-- (18,1.2);
\draw (18,1.2)-- (21,1.2);
\draw (21,1.2)-- (21,0.2);
\draw (9,0.4)-- (6,0.4);
\draw (6,0.4)-- (6,1.4);
\draw (6,1.4)-- (9,1.4);
\draw (9,1.4)-- (9,0.4);
\draw (12,0.6)-- (9,0.6);
\draw (9,0.6)-- (9,1.6);
\draw (9,1.6)-- (12,1.6);
\draw (12,1.6)-- (12,0.6);
\draw (15,0.8)-- (12,0.8);
\draw (12,0.8)-- (12,1.8);
\draw (12,1.8)-- (15,1.8);
\draw (15,1.8)-- (15,0.8);
\draw (18,1)-- (15,1);
\draw (15,1)-- (15,2);
\draw (15,2)-- (18,2);
\draw (18,2)-- (18,1);
\draw (3,1)-- (0,1);
\draw (0,1)-- (0,2);
\draw (0,2)-- (3,2);
\draw (3,2)-- (3,1);
\draw (3,2)-- (0,2);
\draw (0,2)-- (0,3);
\draw (0,3)-- (3,3);
\draw (3,3)-- (3,2);
\draw (6,1.2)-- (3,1.2);
\draw (3,1.2)-- (3,2.2);
\draw (3,2.2)-- (6,2.2);
\draw (6,2.2)-- (6,1.2);
\draw (6,2.2)-- (3,2.2);
\draw (3,2.2)-- (3,3.2);
\draw (3,3.2)-- (6,3.2);
\draw (6,3.2)-- (6,2.2);
\draw (9,1.4)-- (6,1.4);
\draw (6,1.4)-- (6,2.4);
\draw (6,2.4)-- (9,2.4);
\draw (9,2.4)-- (9,1.4);
\draw (12,1.6)-- (9,1.6);
\draw (9,1.6)-- (9,2.6);
\draw (9,2.6)-- (12,2.6);
\draw (12,2.6)-- (12,1.6);
\draw (15,1.8)-- (12,1.8);
\draw (12,1.8)-- (12,2.8);
\draw (12,2.8)-- (15,2.8);
\draw (15,2.8)-- (15,1.8);
\draw (18,2)-- (15,2);
\draw (15,2)-- (15,3);
\draw (15,3)-- (18,3);
\draw (18,3)-- (18,2);
\draw (9,2.4)-- (6,2.4);
\draw (6,2.4)-- (6,3.4);
\draw (6,3.4)-- (9,3.4);
\draw (9,3.4)-- (9,2.4);
\draw (12,2.6)-- (9,2.6);
\draw (9,2.6)-- (9,3.6);
\draw (9,3.6)-- (12,3.6);
\draw (12,3.6)-- (12,2.6);
\draw (15,2.8)-- (12,2.8);
\draw (12,2.8)-- (12,3.8);
\draw (12,3.8)-- (15,3.8);
\draw (15,3.8)-- (15,2.8);
\draw (18,3)-- (15,3);
\draw (15,3)-- (15,4);
\draw (15,4)-- (18,4);
\draw (18,4)-- (18,3);
\draw (3,3)-- (0,3);
\draw (0,3)-- (0,4);
\draw (0,4)-- (3,4);
\draw (3,4)-- (3,3);
\draw (21,1.2)-- (18,1.2);
\draw (18,1.2)-- (18,2.2);
\draw (18,2.2)-- (21,2.2);
\draw (21,2.2)-- (21,1.2);
\draw (21,2.2)-- (18,2.2);
\draw (18,2.2)-- (18,3.2);
\draw (18,3.2)-- (21,3.2);
\draw (21,3.2)-- (21,2.2);
\draw (1.5,0.5) node{0};
\draw (16.5,0.5) node{0};
\draw (1.5,1.5) node{5};
\draw (16.5,1.5) node{5};
\draw (1.5,2.5) node{10};
\draw (16.5,2.5) node{10};
\draw (4.5,0.7) node{1};
\draw (19.5,0.7) node{1};
\draw (4.5,1.7) node{6};
\draw (19.5,1.7) node{6};
\draw (4.5,2.7) node{11};
\draw (19.5,2.7) node{11};
\draw (7.5,0.9) node{2};
\draw (7.5,1.9) node{7};
\draw (7.5,2.9) node{12};
\draw (10.5,1.1) node{3};
\draw (10.5,2.1) node{8};
\draw (10.5,3.1) node{13};
\draw (13.5,1.3) node{4};
\draw (13.5,2.3) node{9};
\draw (13.5,3.3) node{14};
\end{tikzpicture}
    \caption{The translates $R_{i,j}$ of $R^u(C,C)$ used in the proof of \cref{prop:h}\ref{item:2} in the case $C=5$. 
    In each rectangle, we have indicated for which $i$ it is used to occupy the line $l_u(i)$.}
    \label{fig:bricks}
\end{figure}

\begin{proof}[Proof of \ref{item:2}]
In order to upper bound $h^u_p$, it suffices to consider a particular way of occupying all lines of $R^u(m,n)$ for $m$ large enough. Fix a helping set $H\in\cH^u$ and a positive integer $C\in1/\rho_u^2\bbN$ such that $H\subset R^u(C,C)$. Fix some $v\in l_u(1)$ with $\<u^\perp,v\>\ge C$ and consider the disjoint rectangles $R_{i,j}=(iv+jC\rho_uu)+R^u(C,C)$ for $(i,j)\in\bbZ^2$ (see \cref{fig:bricks}). For each $0\le k<n$, let \[m_k=\left|\left\{(i,j)\in\bbZ^2:R_{i,j}\subset R^u(m,\infty),i+Cj=k\right\}\right|,\]
so that $iv+jC\rho_uu + H\subset R^u(m,\infty)$ can occupy line $l_u(k)$. By construction, for some $c>0$ it holds that for all $m$ large enough and $k\ge 0$, $m_k\ge cm$. Then independence yields
\begin{align*}
\bbP_p(\cA^u(m,n))&{}\ge \prod_{k=0}^{n-1}\left(1-\left(1-p^{|H|}\right)^{m_k}\right)\ge (1-(1-p^{\alpha(u)})^{cm})^n\\
&{}\ge\exp\left(-n\log \left(1- e^{-cmp^{\alpha(u)}}\right)\right).\end{align*}
Recalling \cref{eq:def:mu} and that $h^u_p(p^{\alpha(u)m})=-\log \mu$, we recover the second inequality of \cref{estimate} for $x=p^{\alpha(u)}m$. The inequality for arbitrary $x\ge p^{\alpha(u)}/c_u$ with $c_u>0$ small enough then follows since $h_p^u$ is non-increasing.

Turning to the first inequality in \cref{estimate}, we will see, in \cref{lem:helping}, that one can find $C\in1/\rho_u^2\bbN$ large enough so that for any $H\in\cH^u$, there exists $t\in u^\perp\bbZ$ such that $H+t\subset R^u(C,C)$. Then, if $\A^u(m,n)$ occurs, every rectangle of the form $iC\rho_uu +R^u(m,C)$ contained in $R^u(m,n+C)$ must contain an element of $iC\rho_uu+\cH^u$. Since there are at most $C^{2\alpha(u)}$ possibilities for the helping set up to translation, the Harris inequality \eqref{eq:Harris} gives
\[\mathbb{P}_p(\A^u(m,n))\leq \prod_{i=0}^{\lfloor n/C\rfloor}\left(1-\left(1-p^{\alpha(u)}\right)^{C^{2\alpha(u)}m}\right)\le \left(1-e^{-2m(C^2p)^{\alpha(u)}}\right)^{n/C}.\]
The first inequality of \cref{estimate} then follows as above.
\end{proof}

\begin{proof}[Proof of \ref{item:3}] 
Fix $a<b$. Let us prove that $h^u_p$ converges to some function $h^u$ as $p\rightarrow 0$. From \cref{eq:vpm:bounds} we have that for some $C>0$, any $p$ small enough, $n>C$ and $x\ge Cp^{\alpha(u)}$ with $x\in p^{\alpha(u)}\bbN$,
\[\frac{-\log \mathbb P_p(\A^u(xp^{-\alpha(u)},n))}{n+C}\le  h^u_p(x)\le \frac{-\log \mathbb P_p(\A^u(xp^{-\alpha(u)},n))}{n-C}.\]
Since $h_p^u$ was defined by linear interpolation for $x\not\in p^{\alpha(u)}\bbN$, if we also interpolate $\log\bbP_p[\A^u(xp^{-\alpha(u)},n)]$ linearly, the above inequalities remain valid for any $x\ge a$. It is therefore sufficient to prove that for each fixed $n>0$,
$x\mapsto \mathbb P_p(\A^u(xp^{-\alpha(u)},n))$ converges uniformly on $[a,b]$ as $p\rightarrow 0$ to a limit taking values in $(0,1)$. The fact that the limit cannot be $0$ or $1$ and the integrability of $h^u$ follow from \ref{item:2}, while continuity and monotonicity pass through the uniform convergence.

Fix $n\ge C$. For any $E\subseteq \{0,\dots,n-1\}$, define $\A^u(m,n,E)$ to be the event that lines $l_u(i)$ for $i\in E$ are not occupied in $A\cap R^u(m,\infty)$. Via the inclusion-exclusion principle, it is sufficient to show that $x\mapsto \mathbb P_p(\A^u(xp^{-\alpha(u)},n,E))$ converges uniformly on $[a,b]$ for any fixed $E$.

Fix an integer $k\ge C$. Consider the rectangle $R^u(m,n)$ for $m$ divisible by $k$ and partition it into $R^u(m,n)=\bigsqcup_{i=0}^{m/k-1}iku^\perp+R^u(k,n)$. Next, given a configuration $A\subseteq R^u(m,n)$, let $\tau A=(A+u^\perp)/mu^\perp\bbZ\subseteq R^u(m,n)$, that is, the circular shift of $A$ by $u^\perp$. In particular, $\tau^mA=A$. Observe that, by \cref{eq:def:Aumn} and the definition of $C$, if $A\not\in\cA^u(m,n,E)$, then $A\supseteq H+t$ for some $H\in \cH^u$ with $H\subset R^u(C,C)$ and $t\in \bigcup_{e\in E}l_u(e)$. But then for each $i\in\{0,\dots,m/k-1\}$, at least $k-C$ out of the shifts $(\tau^j(H+t))_{j=1}^m$ are contained in $iku^\perp +R^u(k,n)$. Thus, for at least $(k-C)m/k$ values of $j\in\{1,\dots,m\}$, we have $\tau^jA\not\in\cA^u(m,n)$. Since the rectangles in the partition are disjoint, this yields
\begin{equation}
\label{eq:Au:ratio}1\ge \frac{1-(\bbP_p(\cA^u(k,n,E)))^{m/k}}{1-\bbP_p(\cA^u(m,n,E))}\ge \frac{k-C}{k}.\end{equation}

Recall that by definition, for every $x\ge p^{\alpha(u)}$, we have 
\begin{equation}
\label{eq:Au:discretisation}\bbP_p(\cA^u(\lceil xp^{-\alpha(u)}\rceil,n,E))\le \bbP_p(\cA^u(xp^{-\alpha(u)},n,E))\le \bbP_p(\cA^u(\lfloor xp^{-\alpha(u)}\rfloor,n,E)).\end{equation}
Moreover, $\P_p(\A^u(k,n,E))=1-C' p^{\alpha(u)}+O(p^{\alpha(u)+1})$ as $p\to0$, where $C'=C'(u,k,n,E)$ is the number of possible positions of translates of a helping set violating the event. When $p$ goes to 0, this leads to 
\begin{equation}
\label{eq:Au:limit}
\left(\bbP_p(\cA^u(k,n,E))\right)^{(xp^{-\alpha(u)}+O(1))/k}\to e^{-xC'/k}\end{equation}
uniformly on $x\in[a,b]$. Combining \cref{eq:Au:ratio,eq:Au:discretisation,eq:Au:limit}, we get
\begin{equation}
\label{eq:Au:exp:bounds}
e^{-xC'/k}+o(1)\ge \bbP_p(\cA^u(xp^{-\alpha(u)},n,E)\ge e^{-xC'/k}-\frac Ck-o(1)\end{equation}
with $o(1)$ going to $0$ as $p\to 0$ uniformly on $x\in[a,b]$.

The definition of $C'$ readily leads to the further quasi-additivity \[|C'(u,k_1+k_2,n,E)-C'(u,k_1,n,E)-C'(u,k_2,n,E)|\le nC''(u)\]
for a suitable constant $C''(u)>0$. Then the sub-additivity lemma gives the existence of $\lim_{k\to\infty}C'/k=\kappa=\kappa(u,n,E)\in(0,\infty)$. Therefore, \cref{eq:Au:exp:bounds} entails the uniform convergence
\[\bbP_p\left(\cA^u\left(xp^{-\alpha(u)},n,E\right)\right)\to e^{-x\kappa}.\qedhere\]
\end{proof}

While the event $\cA^u(m,n)$ enjoys good approximate multiplicativity properties, it will be more convenient to work with a slightly more artificial version of it following \cite{Hartarsky24univupper}. To introduce it we will need a few more notions.

\begin{definition}[$W$-helping sets]
\label{def:W:helping}
Let $u$ be an isolated stable direction. One can show \cite{Bollobas15}*{Lemma 5.2} that there exists a positive integer $W_u$ and $U_1,U_2\in\cU$ such that $U_1\cup((W_u+1)u^\perp +U_2)\subset\bbH_u\cup (u^\perp\{1,\dots,W\})$. We will call any set of $W_u$ consecutive sites of the form $x+(u^\perp\{1,\dots,W_u\})$ a {\em $W$-helping set} for $l_u(n)$, if $x\in l_u(n)$.    
\end{definition}
In words, a $W$-helping set is an interval of $W_u$ sites on $l_u$ such that, with the help of $\bbH_u$, it immediately infects the next sites on $l_u$, thus propagating infection along $l_u$. Consequently, if $H$ is a $W$-helping set for $l_u$, then $[\bbH_u\cup H]=l_u\cup\bbH_u$. Voracious directions have the property that any helping set together with a half-plane quickly generates a $W$-helping set as ensured in the next statement, which is essentially due to \cites{Hartarsky20a}.
\begin{lemma}[Helping sets generate $W$-helping sets]
\label{lem:helping}
Fix a voracious isolated stable direction $u$. For every $V_u\in\bbN$ large enough the following holds for any $H\in\cH^u$. There exists $t\in u^\perp\bbZ$ such that $H+t\subset R^u(V_u,V_u)$ and the set of sites $A_{\lfloor \sqrt{V_u}\rfloor}$ infected at time $\lfloor\sqrt{V_u}\rfloor$ with initial infections $A_0=\bbH_u\cup (H+t)$ contains a $W$-helping set $H'$ for $l_u$ with $H'\subset u^\perp\{-\lfloor V_u/2\rfloor,\dots,\lfloor 3V_u/2\rfloor\}$.
\end{lemma}
\begin{proof}
The fact that for some $V_u$ we can choose $t$ so that $H+t\subset R^u(V_u,V_u)$ was already used in the proof of \cref{prop:h} and does not require voracity. It is proved in \cite{Hartarsky20a}*{Section 2}. Thus, up to the translation vector $t$, there are finitely many helping sets. Since $u$ is voracious, each $H\in\cH^u$ together with $\bbH_u$ generates a $W$-helping set in finite time. Let $T<\infty$ be the maximal such time. Setting $C=\max\{\|u\|:u\in U,U\in\cU\}$, we have that the $W$-helping set is at distance at most $CT$ from $H$, since $u$ is stable. Taking the maximum of $V_u$, $W_u$, $4C^2$ and $T^2$ yields the desired conclusion.
\end{proof}
The next definition is the more technical version of the event $\cA^u(m,n)$ that we will use.
\begin{definition}[Traversability]
\label{def:traversable}
Fix an isolated stable direction $u$ and positive integer $V_u\ge W_u$ such that for any $H\in\cH^u$ there exists $t\in u^\perp\bbZ$ so that $H+t\subseteq R^u(V_u,V_u)$. Further let $m>2V_u$ and $n\ge 1$ be integers. If $n>V_u$, we say that $R^u(m,n)$ is \emph{traversable} in $A$, if $A-V_uu^\perp\in\cA^u(m-2V_u,n-V_u)$ and $A\cap R^u(m,n)$ contains a $W$-helping set for $l_u(n-i)$ for each $i\in\{1,\dots, V_u\}$. If $n\le V_u$, we say that $R^u(m,n)$ is \emph{traversable} in $A$, if there are $W$-helping sets in $A\cap R^u(m,n)$ for $l_u(i)$ for all $i\in\{0,\dots,n-1\}$. Let $\cT(R^u(m,n))$ denote the event that $R^u(m,n)$ is traversable.
\end{definition}
In words, we require helping sets to be far from the boundary of the rectangle (so that \cref{lem:helping} can be used to generate $W$-helping sets before seeing the boundary) and further ask for $W$-helping sets on the last few lines. The use of $W$-helping sets in this definition is that, contrary to helping sets, they are contained in the line they are used for. This way the occurrence of $\cT(R^u(m,n))$ only depends on $A\cap R^u(m,n)$ and it is not hard to check that $[\bbH_u\cup (A\cap R^u(m,n))]\supset R^u(m,n)$ for any $A\in \cT(R^u(m,n))$.

The Harris inequality \eqref{eq:Harris} and \cref{prop:h}\ref{item:1} yield the following.
\begin{corollary}[Traversability probability]
\label{cor:traversability}
For any isolated stable direction $u$, $V_u$ large enough, $m>3V_u$ and $n>V_u$, we have 
\[p^{WV_u}\exp\left(-h^u_p\left(p^{\alpha(u)}(m-2V_u)\right)n\right)\le \bbP_p(\cT(R^u(m,n)))\le \exp\left(-h^u_p\left(p^{\alpha(u)}m\right)(n-2V_u)\right).\]
\end{corollary}

\subsection{Droplets}
\label{subsec:droplets}
We henceforth assume that $\cU$ is not one-dimensional, that is, there does not exist $u\in\bbZ^2$ such that $U\subset u\bbZ$ for all $U\in\cU$. While helping sets are defined with an infected half-plane in mind, we will systematically have only a finite infected region at our disposal. Our next goal is to define the appropriate geometry for such regions, respecting the update family $\cU$.

We need to consider a particular set of directions related to the update family known as \emph{quasi-stable directions} \cite{Bollobas15}. Namely, let 
\begin{equation}
\label{eq:def:S}
\cS=\left\{u\in S^1:\exists U\in\cU,\exists x\in U:\<x,u\>=0\right\}.
\end{equation}
Note that quasi-stable directions are necessarily rational. We index them $u_1,\dots,u_{|\cS|}$ in counterclockwise order and indices are considered modulo $|\cS|$. Since we will often consider sequences of numbers indexed by $\cS$, we denote by $\be_u$ the canonical basis of $\bbR^\cS$ and use bold letters for vectors in this space. When $\cU$ is isotropic with difficulty $\alpha$ (recall \cref{eq:def:alpha:global}), the set 
\begin{equation}
\label{eq:def:Sa}
\cS_\alpha=\left\{u\in S^1:\alpha(u)=\alpha\right\}\subseteq \cS
\end{equation}
of isolated stable directions of maximal difficulty is also of particular importance. As it will be convenient to work with continuous regions, for $a\in\bbR$, we further set \[\bbH_u(a)=\{x\in\bbR^2:\<x,u\><a\rho_u\}.\] However, whenever referring to the bootstrap percolation process with an initial condition contained in $\bbR^2$, we will mean its intersection with $\bbZ^2$.
\begin{definition}[Droplet]
\label{def:droplet}A \emph{droplet} $D$ is a non-empty set of the form $D=D[{\bf a}]=\bigcap_{u\in \cS} \bbH_u(a_u)$ where ${\bf a}\in \bbR^\cS$ (see \cref{fig:droplet}). Given a droplet $D$, its \emph{radii} $\ba\in\bbR^\cS$ are given by $a_u=\sup_{d\in D}\<d,u/\rho_u\>$, so that $D=D[\ba]$ and we systematically assume sequences defining droplets to be chosen this way. We similarly define $\cS_\alpha$-droplets, replacing $\cS$ by $\cS_\alpha$ and similarly for all subsequent notions involving droplets.
\end{definition}
For $u\in \cS$, define the {\em edge} \[E_u(D[\ba])=\{x\in\bbR^2:\<x,u\>=a_u\rho_u,\forall v\in\cS\setminus\{u\},\<x,v\><a_v\rho_v\},\]
that is, the $u$-side of the polygon $D[\ba]$. Note that $E_u(D)\cap D=\varnothing$ for any droplet $D$. The \emph{dimension} ${\bf m}\in [0,\infty)^\cS$ of $D$ is given by $m_u=|E_u(D)|\rho_u$ for every $u\in \cS$, where $|E_u(D)|$ is the Euclidean length of the edge. The \emph{perimeter} of $D$ with dimension $\bm$ is defined as 
\[\Phi(D)=\sum_{u\in \cS} m_u.\] 
We will require a notion of ``circular'' droplet. For $k\in[0,\infty)$, let $D[k]$ be the symmetric droplet with dimension $(k,\dots,k)\in[0,\infty)^\cS$ (not to be confused with radii). In order to construct $D[k]$, set $x_1=0$ and $x_{i+1}=x_i-ku_i^\perp$. Since $\cS$ is symmetric, we obtain $x_{|\cS|+1}=x_0$ and $D[k]$ is constructed as the polygon with vertices $(x_i)_{i=1}^{|\cS|}$ translated appropriately.

The \emph{location} of $D_1=D_1[{\bf a}]\subseteq D_2=D_2[{\bf b}]$ is given by ${\bf s}=\bb-\ba\in \mathbb [0,\infty)^\cS$. 
The \emph{total location} $\Psi(D_1,D_2)$ is defined by 
\[\Psi(D_1,D_2)=\sum_{u\in \cS}s_u.\]
Note that $\Psi(D_1,D_2)$ does not depend on the positions of $D_1$ and $D_2$, but just on their shapes. 

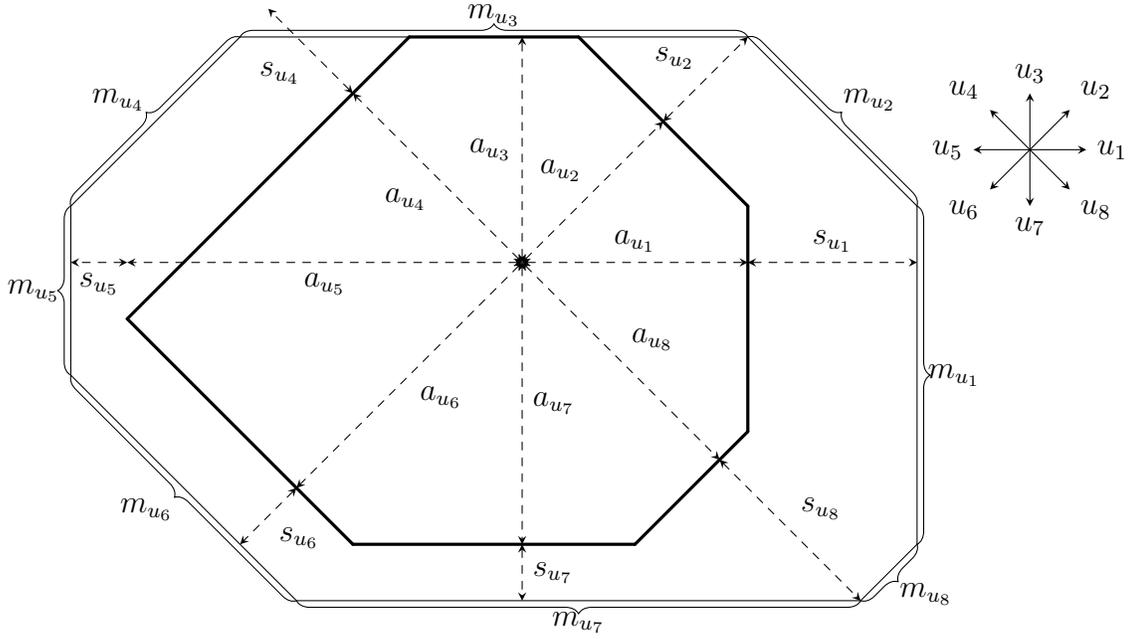
\begin{figure}
	\centering
\begin{tikzpicture}[line cap=round,line join=round,>=stealth,x=0.75cm,y=0.75cm]
\begin{scope}[shift={(9,2)}]
\draw [->] (0,0) -- (1,0) node[right]{$u_1$};
\draw [->] (0,0) -- (0.707,0.707) node[above right]{$u_2$};
\draw [->] (0,0) -- (0,1) node[above]{$u_3$};
\draw [->] (0,0) -- (-0.707,0.707) node[above left]{$u_4$};
\draw [->] (0,0) -- (-1,0) node[left]{$u_5$};
\draw [->] (0,0) -- (-0.707,-0.707) node[below left]{$u_6$};
\draw [->] (0,0) -- (0,-1) node[below]{$u_7$};
\draw [->] (0,0) -- (0.707,-0.707) node[below right]{$u_8$};
\end{scope}
\draw[very thick] (-3,-5)-- (-7,-1);
\draw[very thick] (-7,-1)-- (-2,4);
\draw[very thick] (-2,4)-- (1,4);
\draw[very thick] (1,4)-- (4,1);
\draw[very thick] (4,1)-- (4,-3);
\draw[very thick] (4,-3)-- (2,-5);
\draw[very thick] (2,-5)-- (-3,-5);
\draw (6,-6)-- (7,-5);
\draw[decorate,decoration={brace,amplitude=5pt}] (7,-5) -- (6,-6) node [midway,below right] {$m_{u_8}$};
\draw (7,-5)-- (7,1);
\draw[decorate,decoration={brace,amplitude=5pt}] (7,1) -- (7,-5) node [midway,right] {$m_{u_1}$};
\draw (7,1)-- (4,4);
\draw[decorate,decoration={brace,amplitude=5pt}] (4,4)--(7,1) node [midway,above right] {$m_{u_2}$};
\draw (4,4)-- (-5,4);
\draw[decorate,decoration={brace,amplitude=5pt}] (-5,4)--(4,4) node [midway,above] {$m_{u_3}$};
\draw (-5,4)-- (-8,1);
\draw[decorate,decoration={brace,amplitude=5pt}] (-8,1)--(-5,4) node [midway,above left] {$m_{u_4}$};
\draw (-8,1)-- (-8,-2);
\draw[decorate,decoration={brace,amplitude=5pt}] (-8,-2)--(-8,1) node [midway,left] {$m_{u_5}$};
\draw (-8,-2)-- (-4,-6);
\draw[decorate,decoration={brace,amplitude=5pt}] (-4,-6)--(-8,-2) node [midway,below left] {$m_{u_6}$};
\draw (6,-6)-- (-4,-6);
\draw[decorate,decoration={brace,amplitude=5pt}] (6,-6)--(-4,-6) node [midway,below] {$m_{u_7}$};
\draw [dashed,<->] (0,0)-- (0,-5) node[midway,right]{$a_{u_7}$};
\draw [dashed,<->] (0,-5)-- (0,-6) node[midway,right]{$s_{u_7}$};
\draw [dashed,<->] (0,0)-- (-4,-4) node[midway,below right]{$a_{u_6}$};
\draw [dashed,<->] (-4,-4)-- (-5,-5) node[midway,below right]{$s_{u_6}$};
\draw [dashed,<->] (0,0)-- (-7,0) node[midway,below]{$a_{u_5}$};
\draw [dashed,<->] (-7,0)-- (-8,0) node[midway,below]{$s_{u_5}$};
\draw [dashed,<->] (0,0)-- (-3,3) node[midway,below left]{$a_{u_4}$};
\draw [dashed,<->] (-3,3)-- (-4.5,4.5) node[midway,below left]{$s_{u_4}$};
\draw [dashed,<->] (0,0)-- (0,4) node[midway,left]{$a_{u_3}$};
\draw [dashed,<->] (0,0)-- (2.5,2.5) node[midway,above left]{$a_{u_2}$};
\draw [dashed,<->] (2.5,2.5)-- (4,4) node[midway,above left]{$s_{u_2}$};
\draw [dashed,<->] (0,0)-- (4,0) node[midway,above]{$a_{u_1}$};
\draw [dashed,<->] (4,0)-- (7,0) node[midway,above]{$s_{u_1}$};
\draw [dashed,<->] (0,0)-- (3.5,-3.5) node[midway,above right]{$a_{u_8}$};
\draw [dashed,<->] (3.5,-3.5)-- (6,-6) node[midway,above right]{$s_{u_8}$};
\end{tikzpicture}
	\caption{An example of two droplets $D[\ba]\subseteq D[\bb]$ with $|\cS|=8$. The radii $\ba\in\bbR^{\cS}$, the location $\bs=\bb-\ba$ and the dimension $\bm$ of $D[\bb]$ are indicated. Note that $s_{u_3}$ is not drawn, since it is 0 in this instance. Further note that $a_{u}$ and $s_{u}$ are measured in units of $\rho_u$, while $m_{u}$ is measured in units of $1/\rho_u$ for every $u\in\cS$.}
	\label{fig:droplet}
\end{figure}

Not every ${\bf m}\in \bbR^{\cS}$ necessarily corresponds to the dimension of a droplet. Yet it is easy to verify that if ${\bf m}$ and ${\bf m'}$ are the dimensions of two droplets $D$ and $D'$, then there exists a droplet with dimensions ${\bf m}+{\bf m'}$. In fact it is given by the Minkowski sum of the droplets
\begin{equation}
\label{eq:def:sum}D[\ba]+D[\bb]:=D[\ba+\bb]=\{x+y:x\in D[\ba],y\in D[\bb]\}.
\end{equation}
For any $z>0$ and droplet $D$ we denote $D^z=D+D[z]$. \Cref{eq:def:sum} immediately entails the following important property of sums that will be used frequently.
\begin{observation}\label{loc}Let $D_1\subseteq D_2$ and $D$ be droplets. The location of $D_1+D\subseteq D_2+D$ is equal to the location of $D_1\subseteq D_2$.\end{observation}
\begin{proof}
Let the radii of $D_1,D_2,D$ be $\ba,\bb,\bc$ respectively. Then, the location of $D_1+D=D[\ba+\bc]\subseteq D[\bb+\bc]=D_2+D$ is $(\bb+\bc)-(\ba+\bc)=\bb-\ba$, which is the location of $D_1=D[\ba]\subseteq D[\bb]=D_2$.
\end{proof}

We will require a further operation on droplets.
\begin{definition}[Span of droplets]
\label{def:span}The \emph{span} of droplets $D_1,\dots,D_k$ denoted by $D_1\vee\dots\vee D_k$ is the smallest droplet containing $\bigcup_{i=1}^k D_i$.
\end{definition}
The following important property follows directly from \cref{def:span,eq:def:sum}: one has that $D[\ba_1]\vee\dots\vee D[\ba_k]=D[\ba^{(1)}\vee\dots\vee\ba^{(k)}]$ with $\ba^{(1)}\vee\dots\vee\ba^{(k)}=(\max_{i=1}^ka^{(i)}_u)_{u\in\cS}$. 

\subsection{The sharp threshold constant \texorpdfstring{$\lambda$}{lambda}}
\label{subsec:W}
In the sequel we assume that $\cU$ is isotropic with difficulty $\alpha$. We are now in position to define a functional depending on two droplets, which will quantify the cost of the smaller one growing to become the larger one.
\begin{definition}
\label{def:W}
For two droplets $D\subseteq D'$ with location $\bs$ and such that the dimension of $D$ is $\bm$, let\footnote{\label{foot:degenerate}Here we make the convention that $h^u(0)=h_p^u(0)=\infty$, but if $m_u=s_u=0$, then $h^u(m_u)s_u=h_p^u(p^\alpha m_u)s_u=0$ for any $u\in\cS$ and $p\in(0,1)$.}
\begin{align*}
W_p(D,D')&{}=p^{\alpha}\sum_{u\in \cS_\alpha} h^u_p\left(p^{\alpha}m_u\right)s_u,\\
W(D,D')&{}=\sum_{u\in \cS_\alpha} h^u(m_u)s_u,
\end{align*}
where $h^u_p$ and $h^u$ are defined in \cref{prop:h}. Let $\mathfrak{D}$ be the set of bi-infinite non-decreasing (for inclusion) sequences of droplets $(D_n)_{n\in\mathbb Z}$ such that $\bigcap_{n\in\bbZ}D_n=\{0\}$ and $\bigcup_{n\in\bbZ}D_n=\bbR^2$. For a sequence $\cD=(D_n)_{n\in\bbZ}\in\mathfrak D$, set
\[\cW(\cD)=\frac12\sum_{n\in\bbZ}W(D_n,D_{n+1}).\]
Finally, the sharp threshold constant is given by
\[\lambda=\inf_{\cD\in\mathfrak D}\cW(\cD).\]
We analogously define $\mathfrak D_\alpha$ for $\cS_\alpha$-droplets and set $\lambda_\alpha=\inf_{\cD\in\mathfrak D_\alpha}\cW(\cD)$.
\end{definition}

Let us emphasise that even though droplets are defined with respect to $\cS$, only directions in $\cS_\alpha$ are featured in $W_p$ and $W$. As we will see, this will entail that $\lambda_\alpha=\lambda$. Also note that $W_p$ is simply the rescaled version of $W$, taking into account the scaling from \cref{prop:h} for directions in $\cS_\alpha$.

The definition of $\lambda$ as the minimizer of an energy functional is a typical feature of metastability phenomena. Since the creation of a droplet of critical size is very unlikely, the procedure to create it tends to minimize the energy. Here, the energy takes the special form of work along a certain sequence of droplets. The sequence along which the work is minimized is therefore related to the typical shape of a critical droplet.

\begin{proposition}\label{base lambda}
The constant $\lambda$ belongs to $(0,\infty)$.
\end{proposition}

\begin{proof}
Let us first show that $\lambda>0$. Observe that there exists a constant $c>0$ such that for any sequence of radii $\ba$ and corresponding dimension $\bm$, we have $\max_{u\in\cS}m_u\le c\max_{u\in\cS_\alpha}a_u$, since there are directions of difficulty $\alpha$ in every semicircle. Consider a sequence of droplets $D_n=D[\ba^{(n)}]$ as in \cref{def:W}. Let $n_0$ be the smallest integer such that $\max_{u\in\cS_\alpha} a^{(n_0)}_u\ge B$ for some fixed constant $B>0$ and let $u_0\in\cS_\alpha$ be such that $a^{(n_0)}_{u_0}=\max_{\cS_\alpha}a^{(n_0)}_u$. Then 
\[\sum_{n=-\infty}^{n_0-1} W(D_n,D_{n+1})\ge h^{u_0}\left(m_{u_0}^{(n_0-1)}\right)\sum_{n=-\infty}^{n_0-1}s_{u_0}^{(n)}=h^{u_0}\left(m_{u_0}^{(n_0-1)}\right) a_{u_0}^{(n_0)}\ge h^{u_0}(cB)B>0,\]
since $h^{u_0}$ is non-increasing and positive by \cref{prop:h}.

Turning to $\lambda<\infty$, consider the sequence $\cD=(D[2^n])_{n\in\bbZ}$ and let $D[1]=D[\ba]$. For some constant $c>0$, its energy is given by
\begin{align}
\nonumber\cW(\cD)&{}=\sum_{n\in\bbZ} W(D[2^n],D[2^{n+1}])= \sum_{u\in\cS_\alpha}\sum_{n\in\bbZ}h^u(2^n)2^na_u\\
&{}\le \frac{-1}{c}\sum_{n\in\bbZ}\log\left(1-e^{-c2^n}\right)2^n<\infty,
\label{eq:powers}
\end{align}
using \cref{prop:h}\ref{item:2}.
\end{proof}

\begin{proposition}
\label{prop:lambda:alpha}
We have $\lambda=\lambda_\alpha$.
\end{proposition}
\begin{proof}
Considering $\cS_\alpha$-droplets as droplets whose dimensions $m_u$ are zero for $u\in\cS\setminus\cS_\alpha$ (like the $u_5$-dimension of $D[\ba]$ in \cref{fig:droplet}), it is clear that $\lambda\le\lambda_\alpha$, so it remains to prove the reverse inequality. Fix $\varepsilon>0$ and let $\cD=(D_n)_{n\in\bbZ}\in\mathfrak D$ be such that $\cW(\cD)\le \lambda+\varepsilon$. For each $n\in\bbZ$, let $D'_n$ be the smallest $\cS_\alpha$-droplet containing $D_n$. Observe that for each $n\in\bbZ$ and $u\in\cS_\alpha$ we have $m_u^{(n)}\le {m'_u}^{(n)}$ and ${s_u'}^{(n)}=s_u^{(n)}$, since $\cS\supseteq\cS_\alpha$, where $\bm^{(n)}$ is the dimension of $D_n$ and $\bs^{(n)}$ is the location of $D_n\subseteq D_{n+1}$ and similarly for ${\bm'}^{(n)}$ and ${\bs'}^{(n)}$. Therefore, setting $\cD'=(D'_n)_{n\in\bbZ}$, we get $\cW(\cD')\le \cW(\cD)=\lambda+\varepsilon$, since the functions $h^u$ are non-increasing. Thus, it remains to check that $\cD'\in\mathfrak D_\alpha$. But this is clear: $D'_n\supseteq D_n\to\bbR^2$ as $n\to\infty$ and $D'_n\to\{0\}$ as $n\to-\infty$ since the same holds for $D_n$. Hence, $\lambda_\alpha\le \cW(\cD')\le\lambda+\varepsilon$ for any $\varepsilon>0$ and we are done.
\end{proof}

\subsection{Constants}
In the subsequent sections we will require a number of large and small quantities that will depend on each other. In order to simplify statements and for convenience, we gather them here. First, fix the update family $\cU$, $\alpha$ from \cref{eq:def:alpha:global}, $\cS$ from \cref{eq:def:S}, $\cS_\alpha$ from \cref{eq:def:Sa} and $\lambda$ from \cref{subsec:W} once and for all and allow all other constants to depend on them. Recall the constants $V_u$ and $c_u$ defined in \cref{prop:h,lem:helping} and $W_u$ from \cref{def:W:helping} for isolated stable directions $u$. Since there are finitely many such $u$, we fix uniform constants $W=\max_u W_u$, $c=\min_{u} c_u$ and $V_u=\max V_u$ (chosen once $W$ is fixed, so that \cref{lem:helping} works). We also allow all subsequent constants to depend on $c,V,W$. We introduce the positive constants $C,K,\varepsilon,G,B,L,Z,T$ so that
\[1\ll C,K\ll\frac{1}{\e}\ll G\ll B\ll L\ll \frac{1}{Z}\ll \frac{1}{T}\ll\frac{1}{p}.\]
That is to say, $C$ and $K$ are positive numbers chosen large enough, $\varepsilon$ is positive small enough depending on $C$ and $K$, $G$ is positive chosen large enough depending on $C$, $K$ and $\varepsilon$ and so on. When constants are introduced more locally, they may also depend on $\cU$, $\alpha$, $\cS$, $\cS_\alpha$, $c$, $V$ and $W$, but not on the other quantities above, unless otherwise stated.

\section{Proof of the upper bound}
\label{sec:upper}
In this section, we focus on proving the following upper bound.
\begin{theorem}
\label{th:upper}
    Let $\cU$ be an isotropic voracious update family of difficulty $\alpha$. Then, recalling $\lambda$ from \cref{def:W}, for any $\varepsilon>0$, we have
    \[\lim_{p\to 0}\bbP_p\left(p^\alpha\log\tau<\lambda+\varepsilon\right)=1\]
\end{theorem}
We aim to exhibit a mechanism for infecting large droplets and estimate its probability. For the rest of \cref{sec:upper}, we fix $\cU$ as in \cref{th:upper}.

\subsection{Lower bound on the probability of growth}

For droplets $D_1\subseteq D_2$,  define $\I(D_1,D_2)=\{[(A\cap D_2)\cup D_1]\supseteq D_2\}$ to be the event that $D_1$ plus the infections present in $D_2$ are enough to infect $D_2$. We now bound the probability of $\I(D_1,D_2)$ for two very similar droplets. Recall from \cref{subsec:droplets} that $D^z=D+D[z]$ for a droplet $D$ and $z>0$.

\begin{proposition}
\label{prop:filling:functional}For any droplets $D_1\subseteq D_2\subseteq D[Bp^{-\alpha}]$ satisfying $\Psi(D_1,D_2)\leq Tp^{-\alpha}$, we have
\begin{equation}\bbP_p\left(\cI\left(D^{Zp^{-\alpha}}_1,D^{Zp^{-\alpha}}_2\right)\right)\ge p^{-C}\exp\left(-(1+\e)\frac{W_p\left(D^{Zp^{-\alpha}}_1,D^{Zp^{-\alpha}}_2\right)}{p^{\alpha}}\right),\label{lower bound droplet}\end{equation}
\end{proposition}
\begin{proof}
Consider two droplets $D^{Zp^{-\alpha}}_1=D[\textbf{a}]\subseteq D^{Zp^{-\alpha}}_2=D[\textbf{b}]$ as in the statement. Let $\bs=\bb-\ba$ be the location. We will use the infection mechanism illustrated in \cref{fig:upper bound}. Fix $u\in \cS$ and let 
\begin{align}
\label{eq:def:mutilde}\tilde m_u&{}=\max\{m\in\bbN:\exists t'\in\bbZ^2,t'+R^u(\tilde m_u,s_u)\subseteq D[\ba+\be_u s_u]\setminus D[\ba]\},\\
\label{eq:def:Ru}R^u&{}=t+R^u(\tilde m_u,s_u)\subseteq D[\ba+\be_u s_u]\setminus D[\ba]\end{align}
for some $t\in\bbZ^2$. Note that $D[\ba+\be_us_u]$ is the droplet $D^{Zp^{-\alpha}}_1$ extended in direction $u$, so that its $u$-edge is contained in the one of $D^{Zp^{-\alpha}}_2$, while the other edges of $D[\ba+\be_us_u]$ contain the corresponding edges of $D^{Zp^{-\alpha}}_1$. In words, $R^u$ is the largest rectangle fitting in the trapezoid depicted in \cref{fig:upper bound}, whose height is very small in the present setting. Note that 
\begin{equation}
\label{eq:mu:tilde:mu}BCp^{-\alpha}\ge m_u\ge \tilde m_u\ge m_u-Cs_u\ge m_u-CTp^{-\alpha}\ge Z/(Cp^\alpha),\end{equation}
where $m_u\ge Zp^{-\alpha}$ is the $u$-dimension of $D^{Zp^{-\alpha}}_1$, using that $1/T\gg 1/Z\gg C$.

Recalling \cref{def:traversable}, consider the event 
\[\cE=\bigcap_{u\in\cS}\cT(R^u).\]

\begin{claim}
\label{claim:EI}
We have $\cE\subseteq \cI(D^{Zp^{-\alpha}}_1,D^{Zp^{-\alpha}}_2)$.
\end{claim}
\begin{proof}
Let us first check that $\cE$ implies that $[D^{Zp^{-\alpha}}_1\cup A\cap D^{Zp^{-\alpha}}_2]\supseteq D[\ba+\be_u]$ for any $u\in\cS$ such that $s_u\ge 1$. 

There are two cases to consider, as in \cref{def:traversable}. If $s_u>V$, then $\cT(R^u)$ guarantees a helping set at distance $V_u$ from the boundary of $R^u$ parallel to $u$. Then, by \cref{lem:helping}, this helping set generates a $W$-helping set that is still at distance at least $V_u/2$ from the boundary. By \cref{def:W:helping}, this $W$-helping set infects the entire segment $\bbZ^2\cap(D[\ba+\be_u]\setminus D[\ba])$ possibly up to bounded distance from the boundary. However, it is known (see \cite{Bollobas15}*{Lemma 5.4}) that, thanks to the choice of $\cS$ in \cref{eq:def:S}, in fact the entire segment $\bbZ^2\cap(D[\ba+\be_u]\setminus D[\ba])$ becomes infected. If $s_u\le V$ instead, then $\cT(R^u)$ directly provides us with a $W$-helping set as above and the same conclusion holds. 

Proceeding by induction on $\Psi(D_1,D_2)$, we obtain the desired conclusion.
\end{proof}
By \cref{claim:EI}, it remains to bound the probability of $\cE$. \Cref{cor:traversability} gives
\begin{equation}
\label{eq:filling:1}
\bbP_p\left(\cI\left(D^{Zp^{-\alpha}}_1,D^{Zp^{-\alpha}}_2\right)\right)\ge \bbP_p(\cE)\ge \prod_{u\in\cS}\left(p^{WV}\exp\left(-h^u_p\left(p^{\alpha(u)}(\tilde m_u-2V)\right)s_u\right)\right).\end{equation}

First consider $u\in\cS\setminus\cS_\alpha$, so that $\alpha(u)\le \alpha-1$ (recall \cref{eq:def:Sa,eq:def:alpha:global}). Then, by monotonicity of $h^u_p$, \cref{eq:mu:tilde:mu,estimate}
\begin{equation}
\label{eq:filling:2}
h^u_p\left(p^{\alpha(u)}(\tilde m_u-2V)\right)\le h^u_p(p^{\alpha(u)-\alpha}Z/C)\le\exp\left(-p^{-1/2}\right).\end{equation}
Next, take $u$ such that $u\in\cS_\alpha$, so that $\alpha(u)=\alpha$. Note that $p^\alpha(m_u-\tilde m_u+2V)\le CT$ and $p^\alpha (\tilde m_u-2V)\in[Z/C,BC]$ by \cref{eq:mu:tilde:mu}. But $h^u_p/h^u\to 1$ uniformly on $[Z/C,BC]$ by \cref{prop:h}\ref{item:3}, $h^u$ is uniformly continuous on this interval and $T\ll 1/C,\e,1/B,Z$, so 
\begin{equation}
\label{eq:filling:3}
h^u_p\left(p^{\alpha(u)}(\tilde m_u-2V)\right)\le (1+\e)h^u(p^\alpha m_u).
\end{equation}
Putting \cref{eq:filling:1,eq:filling:2,eq:filling:3} together with \cref{def:W}, completes the proof of \cref{prop:filling:functional}.
\end{proof}

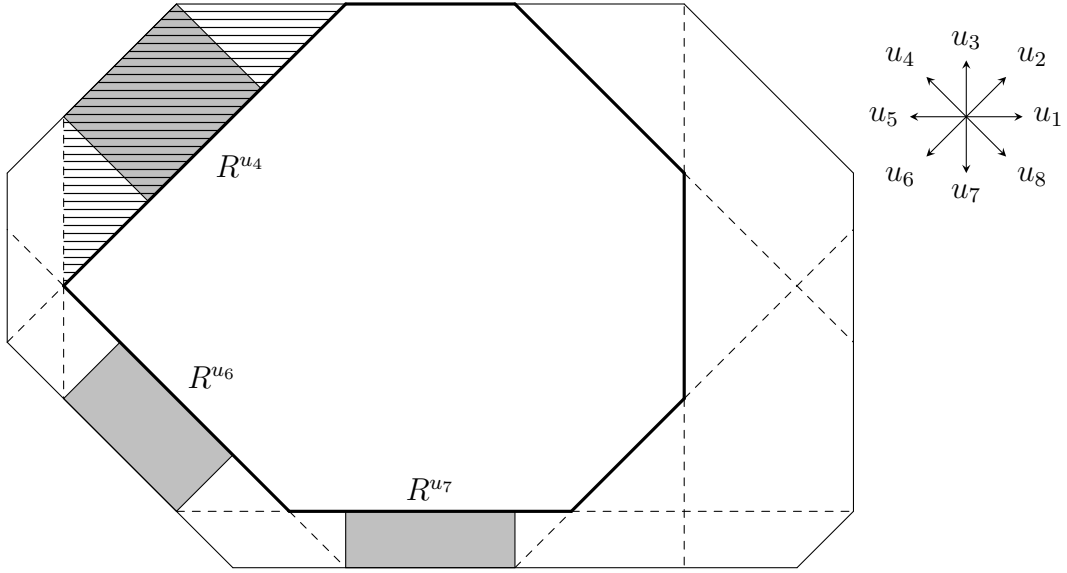
\begin{figure}
	\centering
\begin{tikzpicture}[line cap=round,line join=round,>=stealth,x=0.75cm,y=0.75cm]
\begin{scope}[shift={(9,2)}]
\draw [->] (0,0) -- (1,0) node[right]{$u_1$};
\draw [->] (0,0) -- (0.707,0.707) node[above right]{$u_2$};
\draw [->] (0,0) -- (0,1) node[above]{$u_3$};
\draw [->] (0,0) -- (-0.707,0.707) node[above left]{$u_4$};
\draw [->] (0,0) -- (-1,0) node[left]{$u_5$};
\draw [->] (0,0) -- (-0.707,-0.707) node[below left]{$u_6$};
\draw [->] (0,0) -- (0,-1) node[below]{$u_7$};
\draw [->] (0,0) -- (0.707,-0.707) node[below right]{$u_8$};
\end{scope}
\fill[fill=black,fill opacity=0.25] (-3.5,2.5) -- (-5,4) -- (-7,2) -- (-5.5,0.5) -- cycle;
\fill[fill=black,fill opacity=0.25] (-6,-2) -- (-7,-3) -- (-5,-5) -- (-4,-4) -- cycle;
\fill[fill=black,fill opacity=0.25] (-2,-5) -- (-2,-6) -- (1,-6) -- (1,-5) -- cycle;
\draw(-3.5,2.5) -- (-5,4) -- (-7,2) -- (-5.5,0.5) -- cycle;
\draw (-6,-2) -- (-7,-3) -- (-5,-5) -- (-4,-4) -- cycle;
\draw (-2,-5) -- (-2,-6) -- (1,-6) -- (1,-5) -- cycle;
\draw (-0.5,-5) node[above]{$R^{u_{7}}$};
\draw (-5,-3) node[above right]{$R^{u_{6}}$};
\draw (-4.5,1.5) node[below right]{$R^{u_{4}}$};
\draw[very thick] (-3,-5)-- (-7,-1);
\draw[very thick] (-7,-1)-- (-2,4);
\draw[very thick] (-2,4)-- (1,4);
\draw[very thick] (1,4)-- (4,1);
\draw[very thick] (4,1)-- (4,-3);
\draw[very thick] (4,-3)-- (2,-5);
\draw[very thick] (2,-5)-- (-3,-5);
\fill[pattern=horizontal lines] (-2,4)-- (-5,4)--(-7,2)--(-7,-1)--cycle;
\draw (6,-6)-- (7,-5);
\draw (7,-5)-- (7,1);
\draw (7,1)-- (4,4);
\draw (4,4)-- (-5,4);
\draw (-5,4)-- (-8,1);
\draw (-8,1)-- (-8,-2);
\draw (-8,-2)-- (-4,-6);
\draw (6,-6)-- (-4,-6);
\draw [dashed] (-8,-2)-- (-7,-1);
\draw [dashed] (-7,-1)-- (-8,0);
\draw [dashed] (-7,-1)-- (-7,2);
\draw [dashed] (-7,-1)-- (-7,-3);
\draw [dashed] (-5,-5)-- (-3,-5);
\draw [dashed] (-3,-5)-- (-2,-6);
\draw [dashed] (1,-6)-- (2,-5);
\draw [dashed] (2,-5)-- (7,-5);
\draw [dashed] (4,-3)-- (4,-6);
\draw [dashed] (4,-3)-- (7,0);
\draw [dashed] (4,1)-- (4,4);
\draw [dashed] (4,1)-- (7,-2);
\end{tikzpicture}
	\caption{The rectangles $R^u$ used in the proof of \cref{prop:filling:functional,prop:bound:E} for the droplets from \cref{fig:droplet}. In the picture only 3 of these rectangles are non-empty, as it can be seen thanks to the dashed lines. However, in \cref{prop:filling:functional} it is not possible for any of the rectangles to be empty, since the total location is much smaller than the smallest dimension. The trapezoid $D[\ba+\be_{u_4}s_{u_4}]\setminus D[\ba]$ is hatched.}
	\label{fig:upper bound}
\end{figure}

\subsection{Proof of Theorem~\ref{th:upper}}
We say that a droplet is $D$ \emph{internally filled} if $[D\cap A]\supset D$ and denote the corresponding event by $\cI(D)$. Our next goal is to prove the upper bound of our main result. To that end, we prove lower bounds on the probability of internal filling progressively larger droplets thanks to \cref{prop:filling:functional}. We start by proving that a small droplet is created with fairly good probability. While the next statement can be extracted from the proof of \cite{Hartarsky24univupper}*{Lemma 5.6}, itself inspired by \cite{Bollobas23}*{Section 4}, we include a self-contained proof for completeness.

\begin{lemma}[Subcritical growth]
\label{creation of a small seed}
We have
\[\P_p(\cI(D[1/(Bp^{\alpha})]))\geq \exp(-\e p^{-\alpha}).\]
\end{lemma}

\begin{proof}
To see this, we will proceed similarly to the proof of \cref{prop:filling:functional}. Fix a constant $\kappa>1$ close enough to 1 and $N\in\bbN$ such that $C\kappa^N=1/(Bp^{\alpha})$. For $n\in\{0,\dots,N\}$, let $D_n=D[C \kappa^n]$. In order for the final droplet $D_N$ to be internally filled, it suffices for the first one to be fully infected and all events $\cI(D_i,D_{i+1})$ to occur. As in the proof of \cref{prop:filling:functional}, in order to guarantee the latter, it suffices for suitable translates of the rectangles $R^u(\kappa^i,(\kappa^{i+1}-\kappa^i)a_u)$ to be traversable, where $D_0=D[\ba]$. 

Therefore, the independence of these events, \cref{cor:traversability} and \cref{prop:h}\ref{item:2} give
\begin{align*}\bbP_p[\cI(D[1/(Bp^{\alpha})])]\ge{}& p^{C^3+NWV|\cS|}\prod_{i=0}^{N-1}\prod_{u\in\cS}\exp\left(\left(\kappa^{i+1}-\kappa^i\right)a_u\log\left(1-e^{-p^{\alpha(u)}\kappa^i}\right)\right)\\
\ge{}& e^{C\log^2(1/p)}\exp\left(\sum_{i=0}^{N-1} \sum_{u\in\cS}C\kappa^i\log\left(1-e^{-p^{\alpha(u)}\kappa^i}\right)\right).
\end{align*}
The terms corresponding to $u\in\cS\setminus\cS_\alpha$ contribute a negligible factor $\exp(-C^2p^{-\alpha(u)})$. On the other hand, terms with $u\in\cS_\alpha$ can be bounded by
\[\exp\left(-p^{-\alpha}\e /\left(2|\cS_\alpha|\right)\right),\]
since $B$ is large enough depending on $C$ and $\varepsilon$. Putting these bounds together, we obtain the desired result.
\end{proof}

Before we turn to `critical' droplet sizes, which is the most important scale, we will need a truncation and refinement statement for the threshold constant $\lambda$ from \cref{def:W}.
\begin{lemma}
\label{lem:sequence}
There exists a sequence of droplets $(D_n)_{n\leq N}$ such that:\begin{itemize}
\item $D_0^Z\subseteq D[1/B]$,
\item $D[B]\subseteq D^Z_N\subseteq D[L/2]$,
\item $\Psi(D_n^Z,D_{n+1}^Z)\le T$ for every $0\le n\le N-1$,
\item $\sum_{n=0}^{N-1} W(D_n^Z,D_{n+1}^Z)\leq 2\lambda+\e$.
\end{itemize}
\end{lemma}
\begin{proof}
In order to deduce the existence of $(D_n)_{n\le N}$ from \cref{def:W}, we proceed as follows. We start with a sequence $\cD\in\mathfrak D$ such that $\cW(\cD)\le \lambda +\e/3$, so that $\cD$ does not depend on $B$, but only on $\varepsilon$. Note that along this sequence if there is a dimension $m_u^{(n)}=0$ for $u\in\cS_\alpha$, then $a_u^{(n+1)}-a_u^{(n)}=0$, since otherwise $\cW(\cD)$ would be infinite (recall Footnote \ref{foot:degenerate}). We truncate and index the sequence so that its first term is $D_0\subseteq D[1/(2B)]$ and its last one is $D_N\supseteq D[B]$. Since $L$ can be chosen large enough depending on $\cD$, we can ensure that $D_N^Z\subseteq D[L/2]$ and that $m_u^{(i)}\ge 1/L$ for all $u\in\cS_\alpha$ and $i\in\{0,\dots,N-1\}$ such that $a_u^{(n+1)}-a_u^{(n)}\neq0$. Note that since $Z<1/L$, we have
\[0\le \sum_{n=0}^{N-1} \left(W\left(D_n,D_{n+1}\right)-W\left(D^Z_n,D^Z_{n+1}\right)\right)\le \omega(Z)|\cS_\alpha|\max_{u\in\cS}a_u L,\]
where $\omega$ is the maximum of the moduli of continuity of all $h^u$ over the compact set $[1/L,L]$ and $D[1]=D[\ba]$. The right-hand side above goes to $0$ uniformly in the choice of the sequence as $Z\to 0$ with $L$ fixed. 

It therefore remains to show that we can refine the sequence in order to have $\Psi(D_n,D_{n+1})\le T$ (recall \cref{loc}). Let $D_n=D[\ba^{(n)}]$ and $D_{n+1}=D[\ba^{(n+1)}]$. We create a sequence of intermediate droplets $D_n=D^{n,0}\subset D^{n,1}\subset\dots\subset D^{n,k+1}=D_{n+1}$ recursively as follows. Given $i\ge 0$ and the droplet $D^{n,i}=D[\ba^{n,i}]$, let $u^{n,i}$ be an arbitrarily chosen direction such that its dimension $m^{n,i}_{u^{n,i}}$ is not smaller than $m^{(n)}_{u^{n,i}}$ and $a^{(n+1)}_{u^{n,i}}-a^{n,i}_{u^{n,i}}\neq 0$. The fact that such a direction $u$ necessarily exists whenever $D_n\subseteq D^{n,i}\subset D_{n+1}$ follows from the classical fact that for two convex sets $A\subseteq B$, the perimeter of $A$ is at most the perimeter of $B$. Having fixed $u^{n,i}$, let 
\[\ba^{n,i+1}=\ba^{n,i}+\be_{u^{n,i}}\min\left(T,a^{(n+1)}_{u^{n,i}}-a^{n,i}_{u^{n,i}}\right)\] and define $D^{n,i+1}=D[\ba^{n,i+1}]$.

Clearly, the refined sequence satisfies the first three conditions of \cref{lem:sequence}. It therefore remains to check that
\[W(D_n,D_{n+1})\ge \sum_{i=0}^kW(D^{n,i},D^{n,i+1}).\]
To see this, recall \cref{def:W} and note that $h^u(m_u^{n,i})\le h^u(m_u^{(n)})$ whenever $a^{n,i}_u\neq a^{n,i+1}_u$ and $\sum_{i=0}^k (\ba^{n,i+1}-\ba_u^{n,i})=\ba^{(n+1)}-\ba^{(n)}$.
Thus, the existence of the sequence claimed is established.
\end{proof}
Equipped with \cref{lem:sequence}, we are ready to prove a bound on the critical growth probability.
\begin{proposition}[Critical growth]\label{creation of a critical droplet}
There exists a droplet $D[Bp^{-\alpha}]\subseteq D_p\subseteq D[Lp^{-\alpha}]$ with 
\[\P_p\left(\cI\left(D_p\right)\right)\geq \exp\left(-(2\lambda+C\e)/p^{\alpha}\right).\]
\end{proposition}
\begin{proof}
Let $D_n=D[\ba^{(n)}]$ for $n\in\{0,\dots,N\}$ be the droplets provided by \cref{lem:sequence}. Define the rescaled droplets $(D_n^Z)_p=D_n^{Zp^{-\alpha}}[\mathbf{a}^{(n)}p^{-\alpha}]$. Since $D_0^Z\subseteq D[1/B]$, we have
\[\cI\left(\left(D_N^Z\right)_p\right)\supseteq\cI\left(D[1/(Bp^\alpha)]\right)\cap\bigcap_{n=0}^{N-1}\I\left(\left(D^Z_n\right)_p,\left(D^Z_{n+1}\right)_p\right).\]
Therefore, the Harris inequality gives
\begin{align*}
\P_p\left(\cI\left(\left(D_N^Z\right)_p\right)\right)\geq{}& \P_p\left(\cI\left(D\left[1/(Bp^{\alpha})\right]\right)\right)\prod_{n=0}^{N-1}\P_p\left(\I\left(\left(D^Z_n\right)_p,\left(D^Z_{n+1}\right)_p\right)\right)\\
\geq{}& \exp\left(-\e/p^{\alpha}\right)\prod_{n=0}^{N-1} p^{-C}\exp \left(-(1+\e)\frac{W_p((D_n^Z)_p,(D_{n+1}^Z)_p)}{p^{\alpha}}\right)
\end{align*}
by \cref{prop:filling:functional,creation of a small seed}. We next apply \cref{prop:h}\ref{item:3} and \cref{def:W} and note that $N$ does not depend on $p$, to obtain
\begin{align*}
\P_p\left(\cI\left(\left(D_N^Z\right)_p\right)\right)\geq{}&\exp(-2\e/p^{\alpha})\prod_{n=0}^{N-1}\exp \left(-(1+\e)^2\frac{W(D_n^Z,D_{n+1}^Z)}{p^{\alpha}}\right)\\
\geq{}& \exp\left(-\frac{2\e+(1+\e)^2(2\lambda+\e)}{p^{\alpha}}\right),\end{align*}
further using the fourth property of the sequence $(D_n)$ in \cref{lem:sequence}. The desired result follows since $C$ is large and $\varepsilon$ small enough.
\end{proof}

Once we are past the critical scale, growth becomes easy, as shown by the following result.
\begin{corollary}[Supercritical growth]
\label{creation of a large droplet}We have
\[\P_p\left(\cI\left(D\left[p^{-3W}\right]\right)\right)\geq \exp\left(-(2\lambda+2C\e)/p^{\alpha}\right).\]
\end{corollary}

\begin{proof}
\Cref{creation of a critical droplet} implies that there exists a droplet $D_p=D[\ba]\supseteq D(Bp^{-\alpha})$ and
\[
\P_p(\cI(D_p))\geq \exp(-(2\lambda+C\e)/p^{\alpha}).\]
We may then proceed as in the proof of \cref{creation of a small seed}, growing the droplet dimensions exponentially. This leads to 
\begin{multline*}
\bbP_p\left(\cI\left(D\left[p^{-3W}\right]\right)\right)\\\ge \exp\left(-\frac{2\lambda +C\e}{p^{\alpha}}\right)p^{NWV|\cS|}\prod_{i=0}^{N-1}\exp\left(\frac{|\cS|(\kappa^{i+1}-\kappa^i)CB}{p^\alpha}\log\left(1-e^{-\kappa^iB/C}\right)\right),\end{multline*}
where $\kappa>1$ is a constant close enough to 1 and we assumed for simplicity that $p^{-3W}=\kappa^N Bp^{-\alpha}$ for some integer $N$. Taking $B$ large, the above product can be made larger than $\exp[-\e/p^{\alpha}]$ and we have that $N$ is logarithmic in $1/p$, so the conclusion follows.
\end{proof}

Finally, we can conclude the proof of the upper bound of \cref{th:main} in the usual way following \cite{Aizenman88}. The idea is as follows. We consider disjoint translates of $D[p^{-3W}]$ up to such a distance from the origin that it becomes likely for one to be internally filled in view of \cref{creation of a large droplet}. Furthermore, we show that it is likely that all segments of length $p^{-3W}$ in this volume contain a $W$ helping set. These $W$-helping sets are sufficient to allow the droplet to grow until it reaches the origin.
\begin{proof}[Proof of \cref{th:upper}]
Fix $\Lambda=\exp((\lambda+2C\e)/p^{\alpha})$. Let $\cE$ be the event that for all $u\in\cS$, every translate of the rectangle $R^u(p^{-3W},1)$ included in $D[\Lambda]$ contains a $W$-helping set. The probability of this event can be bounded from below by
\[\P_p(\E)\geq \left(1-\left(1-p^{W}\right)^{\lfloor p^{-3W}/W\rfloor}\right)^{|\cS|\cdot|D[\Lambda]\cap\bbZ^2|}\to 1.\]

Denote by $\mathcal F$ the event that there exists a translate of $D[p^{-3W}]$ included in $D[\Lambda]$ which is internally filled. Applying \cref{creation of a large droplet} and fitting $(\Lambda p^{3W})^2/C$ disjoint translates of $D[p^{-3W}]$ into $D[\Lambda]$, one obtains
\[\P_p(\mathcal F)\geq 1-\left(1-\exp[-(2\lambda+2C\e)/p^{\alpha}]\right)^{(\Lambda p^{3W})^2/C}\to 1.\]

Moreover, the simultaneous occurrence of $\E$ and $\mathcal F$ implies that $p^{\alpha}\log \tau\leq \lambda+3C\e$ for $p$ small enough. Indeed, each site in the internally filled translate of $D[p^{-3W}]$ granted by $\cF$ becomes occupied in time at most $|D[p^{-3W}]\cap\bbZ^2|$, since at least one new site becomes infected at each step. After the creation of this supercritical droplet, it only takes a time of order $p^{-3W}\Lambda$ to progress and reach 0, thanks to the event $\E$. More precisely, growing one of the radii of our droplet by 1 only requires a time of order $p^{-3W}$ regardless of its size, since each $W$-helping set grows linearly along its edge (recall \cref{def:W:helping}). The Harris inequality yields
\[\P_p\left(p^{\alpha}\log T\leq \lambda+3C\e\right)\geq \P_p(\E\cap\mathcal F)\geq \P_p(\E)\P_p(\mathcal F)\rightarrow 1\]
which concludes the proof of \cref{th:upper}, since $C\varepsilon\ll 1$.
\end{proof}

\section{Proof of the lower bound}
\label{sec:lower}
We next turn to the lower bound.
\begin{theorem}
    \label{th:lower}
    Let $\cU$ be a symmetric isotropic update family of difficulty $\alpha$. Then, recalling $\lambda$ from \cref{def:W}, for any $\varepsilon>0$, we have
    \[\lim_{p\to 0}\bbP_p\left(p^\alpha\log\tau>\lambda-\varepsilon\right)=1.\]
\end{theorem}
Our aim is to control all possible ways of creating large droplets. For the rest of \cref{sec:lower}, we fix $\cU$ as in \cref{th:lower}.
\subsection{Upper bound on the probability of growth}
Since the process is not obliged to form droplets, but could instead use more complicated shapes, we will need some further notions to suitably reduce them to droplets.
\begin{definition}[$\Delta$-connected]
Given $\Delta>0$, we say that a set $X\subseteq\bbZ^2$ is \emph{$\Delta$-connected} if it is connected in the graph $\Gamma=(\bbZ^2,\{\{x,y\}:\|x-y\|\le \Delta\}\})$.
\end{definition}
It is known that there exists a constant $K=K(\cU)>0$ such that for all stable directions $u$ and all sets $S\subset \bbZ^d$ such that $S\not\in\cH^u$ and $|S|\le \alpha(u)$, we have
\begin{equation}
\label{eq:K:1}\max\left\{d(x,S):x\in [S\cup\bbH_u]\setminus\bbH_u\right\}<K/3\end{equation}
(see \cite{Hartarsky20a} for an explicit bound on $K$). In particular, applying this to both $u$ and $-u$, we see that for any $S\subset\bbZ^d$ such that $|S|<\alpha(u)$ we have 
\begin{equation}\label{eq:K:2}\max\{d(x,S):x\in[S]\}<K/3.\end{equation}
We further assume $K$ large enough so that for any stable $u$ and any $S\in \cH^u$ we have $\diam(S)<K/3$ and $\max\{\|x\|:x\in\bigcup_{U\in\cU}U\}<K/3$.

\begin{definition}[Spanning]
\label{def:spanning}
For two $\cS_\alpha$-droplets $D_1\subseteq D_2$, let $\E(D_1,D_2)$ be the event that there exists a $K$-connected set $X\subseteq [(A\cap D_2)\cup D_1]$ such that every $\cS_\alpha$-droplet containing $X$ also contains $D_2$.

We further write $\cE(D)=\cE(\varnothing,D)$ for any $\cS_\alpha$-droplet $D$ and say that $D$ is \emph{spanned} when $\cE(D)$ occurs.
\end{definition}
Spanning events $\cE$ will play a similar role to the filling events $\cI$ used for the upper bound in \cref{sec:upper}, so our first step is again to link them to the function $W$ from \cref{def:W}. The next proposition makes use of our symmetry assumption.
\begin{proposition}
\label{prop:bound:E}
For any $\cS_\alpha$-droplets $D_1\subseteq D_2$ satisfying $\Phi(D_2)\le CBp^{-\alpha}$ and $\Psi(D_1,D_2)\leq Tp^{-\alpha}$, we have
\begin{equation}
    \P_p\left(\E(D_1,D_2)\right)\leq C\exp \left(-(1-\e)^2\frac{W_p(D_1^{Zp^{-\alpha}},D_2^{Zp^{-\alpha}})}{p^{\alpha}}\right).\label{upper bound droplet}
\end{equation}
\end{proposition}
The idea behind the proof is as follows, roughly following \cite{Holroyd03}. Our goal is to bound the probability of $\cE(D[\ba],D[\bb])$ for $\bb\ge\ba$ very close. Thinking of $D[\ba]$ with not very small dimensions, we can cut $D[\bb]\setminus D[\ba]$ into rectangles like the shaded ones in \cref{fig:upper bound} and some remaining small leftover region. We will treat the leftover region as `boundary condition', since it could, in principle, help for growing the small droplet in multiple directions simultaneously. Yet, it is small, so it is unlikely to find a lot of infections there. Once the boundary condition is fixed, the rectangles become independent and each of them needs to be `crossed' separately. Before turning to the proof of \cref{upper bound droplet}, let us first discuss a lemma, which deals with crossing one of these trapezoids. 

For any $m,n\in \mathbb N$, define the {\em strip}
\[S^u(n)=\{x\in \Z^2:0\le\langle x,u\rangle< n\rho_u\}=\bigcup_{i=0}^{n-1}l^u(i).\]
Also, consider the \emph{crossing} events
\begin{multline}
\label{eq:def:CumnE}\mathcal C^u(m,n,E)\\=\left\{l^u(0)\text{ and }l^u(n)\text{ $K$-connected in }[(A\cap R^u(m,n))\cup (\Z^2\setminus S^u(n))\cup E]\right\},\end{multline}
where $E\subseteq \Z^2\setminus R^u(m,n)$ is viewed as a ``boundary condition''. For such a set $E$, define $J_E$ to be the number of $j\in\{0,\dots, n-1\}$ such that $l_u(j)$ is at distance at most $3K$ from a $3K$-connected set of cardinality $\alpha(u)$ in $E$.

\begin{lemma}\label{upper bound half-plane}Let $u$ be a stable direction. For $m\in[Tp^{-\alpha(u)},CBp^{-\alpha(u)}]$, $n\ge 1/T$ and $E\subseteq \Z^2\setminus R^u(m,n)$, we have
\[\P_p\left(\mathcal C^u(m,n,E)\right)\leq \exp\left(-(1-\e)h_p^u\left(p^{\alpha(u)}m\right)\left(n-LJ_E\right)\right).\]\end{lemma}
\begin{proof}
We prove the result by slicing the rectangle $R^u(m,n)$ into rectangles of fixed (but large) height $k=L/3$, which we assume divisible by $1/\rho_u^2$.

We start by crudely bounding the probability of an `error' event. Let $\E^u(m,k)$ be the event that $A\cap R^u(m,k)$ contains a $3K$-connected set of size $\alpha(u)+1$ or there is a site $a\in A\cap R^u(m,k)$ such that $\<a,u^\perp\>\in[0,3K)\cup[m-3K,m)$. In words, $\cE^u(m,k)$ occurs if there is an unexpectedly large cluster of infections or there is an infection close to the boundary of $R^u(m,k)$.
\begin{claim}
\label{Emlu}
In this setting, $\bbP_p(\cE^u(m,k))\le \exp(-h^u_p(p^{\alpha(u)}m)k)$.
\end{claim}
\begin{proof}
The number of $3K$-connected sets of size $\alpha(u)+1$ in $R^u(m,k)$ is bounded by $Mkm$ for some constant $M=M(K)>0$. Similarly, 
\[\left|\left\{a\in R^u(m,k):\<a,u^\perp\>\in[0,3K)\cup[m-3K,m),\<a,u\>\in[0,\rho_uk)\right\}\right|\le 6Kk.\]
Therefore, since $mp^{\alpha(u)}\le CB$, a union bound gives
\[\P_p\left(\E^u(m,k)\right)\leq Mkm p^{\alpha(u)+1}+6Kkp\le k(MCB+6K)p.
\]
Further recalling the uniform bound from \cref{prop:h}\ref{item:2} and taking $p$ small enough, this gives
\[\bbP_p(\cE^u(m,k))\le \exp\left(-h^u_p(T)k\right)
\leq \exp\left(-h^u_p\left(p^{\alpha(u)}m\right)k\right),\]
also taking into account that $h_p^u$ is non-increasing and $T\le p^{\alpha(u)}m$.
\end{proof}
With \cref{Emlu} at hand, we next prove \cref{upper bound half-plane} in a specific case.
\begin{claim}
\label{eq:ar}
For any $E$ with $J_E=0$,
\begin{equation}\bbP_p\left(\mathcal C^u(m,k,E)\right)\leq \exp\left(-(1-2\e/3)h_p^u\left(p^{\alpha(u)}m\right)k\right).\end{equation}    
\end{claim}
\begin{proof}    
Note that $k>3K(\alpha(u)+1)$. Let us assume in the following that $\E^u(m,k)$ does not occur. Then, $A'=(A\cap R^u(m,k))\cup E$ consists of $3K$-connected components of size at most $\alpha(u)$ contained entirely in $R^u(m,k)$ and $3K$-connected components of size at most $\alpha(u)-1$ contained entirely in $E$ (since $J_E=0$ and $\cE^u(m,k)$ does not occur). 

Consider a $3K$-connected component $\cK$ of $A'$. If $\cK\subseteq E$, then $|\cK|\le \alpha(u)-1$ and $d(\cK,A'\setminus\cK)>3K$. The component $\cK$ cannot be simultaneously close to $\bbH_u$ and $\bbH^{-u}(1-k)$. We claim that for any $x\in[\cK\cup(\bbZ^2\setminus S^u(k)]\cap S^u(k)$ we have $d(x,\cK)< K/3$. Indeed, depending on whether $\cK$ is close to one boundary of $S^u(k)$, the other boundary, or neither boundary, this follows from either \cref{eq:K:1} or \cref{eq:K:2}. 

Let us make the further assumption that neither $l^u(0)$ nor $l^{-u}(-k+1)$ is occupied (recall \cref{def:occupied}), using the notation $l^{-u}(i)$ in order to specify that the line must be occupied in direction $-u$. Then, the same reasoning as above applies to any $3K$-connected component $\cK\subseteq A'\cap R^u(m,n)$ as well, using \cref{eq:K:1}. But then \[\left[A'\cup(\bbZ^2\setminus S^u(k)\right]\cap S^u(k)=\bigsqcup_\cK \left([\cK\cup S^u(k)]\cap S^u(k)\right),\]
since the sets in the disjoint union are at distance at least $3K-2K/3$ apart and therefore do not interact. Thus, each $K$-connected component of $[A'\cup (\bbZ^2\setminus S^u(k))]\setminus S^u(k)$ is generated by a single $3K$-connected component of $A'$, so it cannot $K$-connect $l^u(0)$ to $l^u(k)$. In conclusion, if $\E^u(m,k)$ does not occur, $l^u(0)$ or $l^{-u}(-k+1)$ must be occupied in order for $\cC^u(m,k,E)$ to occur.
 
Depending on whether $l^u(0)$ or $l^{-u}(-k+1)$ is occupied, one can repeat the above argument inductively for the strip $S^u(k)\setminus l^u(0)$ or $S^u(k)\setminus l^{-u}(-k+1)$. This procedure continues as long as the remaining strip is sufficiently thick. Namely, we deduce that for $N=3K(\alpha(u)+1)$, if $\cC^{u}(m,k,E)\setminus\cE^u(m,k)$ occurs, then there exists $k'$ between 0 and $k$ such that $l^u(0),\dots ,l^u(k'-1)$ and $l^{-u}(-k'+N),\dots,l^{-u}(-k+1)$ are occupied in $A'$. 

Set $\P_p(\A^u(m,k))=1$ for $k<0$. By the above, since lines at a distance greater than $K$ are independently occupied, we get
\begin{align}
\nonumber\P_p\left(\mathcal C^u(m,k,E)\right)&{}\leq \P_p\left(\E^u(m,k)\right)+\sum_{k'=0}^{k} \P_p\left(\A^u(m,k')\right)\P_p\left(\A^{-u}(m,k-k'-N)\right)\\
\label{eq:CumkE}&\leq (k+2)\exp \left(-h^u_p\left(p^{\alpha(u)}m\right)(k-N)\right),\end{align} by \cref{Emlu}, symmetry and \cref{prop:h}\ref{item:1} for the second inequality. Recalling that $mp^{\alpha(u)}\le CB$, $3k=L\gg B,1/\varepsilon,C,K$ and \cref{prop:h}, we get
\[k+2\le \exp\left( h^u_p(CB) (k-N)\e/3\right)\le\exp\left(h^u_p\left(p^{\alpha(u)}m\right)(k-N)\varepsilon/3\right).\]
Combining this with \cref{eq:CumkE} and $k-N\ge (1-\e/3)k$, we deduce
\[\P_p\left(\mathcal C^u(m,k,E)\right)\leq \exp \left(-(1-2\e/3)h^u_p\left(p^{\alpha(u)}m\right)k\right).\qedhere\]
\end{proof}
We can now conclude the proof of \cref{upper bound half-plane}. Divide the rectangle $R^u(m,n)$ into $\left\lfloor n/k\right\rfloor$ translates of $R^u(m,k)$ and a remainder of height at most $k$. Then, at least $\left\lfloor n/k\right\rfloor-2J_E$ of these translated rectangles satisfy the condition of \cref{eq:ar}. We thus obtain
\begin{align*}\P_p\left(\mathcal C^u(m,n,E)\right)&\leq \left(\P_p\left(\mathcal C^u(m,k,E)\right)\right)^{\left\lfloor n/k\right\rfloor-2J_E}\\&\le \exp \left(-(1-2\e/3)h^u_p\left(p^{\alpha(u)}m\right)\left(k\lfloor n/k\rfloor-2kJ_E\right)\right)\\&\le \exp \left(-(1-\e)h^u_p\left(p^{\alpha(u)}m\right)\left(n-3kJ_E\right)\right)\end{align*}
for $n\ge 1/T$. This concludes the proof.
\end{proof}
We are now in a position to prove \cref{prop:bound:E}.
\begin{proof}[Proof of \cref{prop:bound:E}] Consider two droplets $D_1=D[\textbf{a}]\subseteq D_2=D[\textbf{b}]$ with $\Phi(D_2)\le CBp^{-\alpha}$ and $\Psi(D_1,D_2)\leq Tp^{-\alpha}$. Let $\bm$ be the dimension of $D_1$ and $\bs$ be the location of $D_1\subseteq D_2$. For each $u\in\cS_\alpha$ we define $R^u$ as in the proof of \cref{prop:filling:functional}, namely, let $R^u$ be the translate of the largest rectangle $R^u(\tilde m_u,s_u)$ such that $R^u\subseteq D[\ba+\be_u s_u]\setminus D[\ba]$ (recall \cref{fig:upper bound,eq:def:Ru,eq:def:mutilde}). Let $x_u\in\bbZ^2$ be such that $R^u=x_u+R^u(\tilde m_u,s_u)$. We set $\bar R^u=R^u$ if $\tilde m_u\ge Tp^{-\alpha}$ and $\bar R^u=\varnothing$ otherwise and let $\bar m_u=\tilde m_u$ if $\tilde m_u\ge Tp^{-\alpha}$ and $\bar m_u=0$ otherwise. Let 
\[X=D_2\setminus\Big(D_1\cup\bigcup_{u\in\cS_\alpha} \bar R^u\Big)\]
be the \emph{leftover} region (in \cref{fig:upper bound} this corresponds to the larger droplet without the smaller one and those among the shaded rectangles which are not too thin). The leftover may have a rather complicated shape, but is, crucially, small. 

Conditioning on $A\cap X$ and recalling \cref{def:spanning,eq:def:CumnE}, we get
\[\bbP_p(\cE(D_1,D_2))\le \bbE_p\Big[\prod_{u\in\cS_\alpha}\bbP_p\left((A-x_u)\in \cC^u(\bar m_u,s_u,(A\cap X)-x_u)|A\cap X\right)\Big].\]
In words, each rectangle $\bar R^u$ is crossed with boundary condition given by the infections in the leftover region. Note that this probability is simply an indicator function for $u$ such that $\tilde m_u<T p^{-\alpha}$, since it is measurable with respect to the conditioning.

Following \cref{upper bound half-plane}, for each $u\in\cS_\alpha$ let $J^u_{A\cap X}$ be the number of $j\in\{0,\dots,s_u-1\}$ such that $l_u(j)+x_u$ is at a distance at most $3K$ from a $3K$-connected set of cardinality $\alpha$ in $A\cap X$. Then, \cref{upper bound half-plane} gives 
\begin{equation}
\label{eq:ED1D2}\bbP_p(\cE(D_1,D_2))\leq\bbE_p\Big[\prod_{\substack{u\in\cS_\alpha\\\bar m_u=0}}\1_{J_{A\cap X}^u=s_u}\exp\Big(-(1-\e)\sum_{\substack{u\in\cS_\alpha\\\bar m_u\neq0}}h^u_p\left(p^\alpha\bar m_u\right)\left(s_u-L J^u_{A\cap X}\right)\Big)\Big],\end{equation}
which becomes an expectation just over the $(J^u_{A\cap X})_{u\in\cS_\alpha}$.

We argue that for each $u$ either $J^u_{A\cap X}$ is small enough not to perturb $s_u$ much or it is large, which is unlikely by itself. Indeed, denoting by $\bm^Z$ the dimension of $D_1^{Zp^{-\alpha}}$, we can bound \cref{eq:ED1D2} from above by
\[\sum_{S\subset \{u\in\cS_\alpha:\bar m_u\neq 0\}}\bbP_p\left(\forall u\in \cS_\alpha\setminus S, J^u_{A\cap X}>\e s_u/L\right)\exp\Big(-(1-\e)^2\sum_{u\in S}h^u_p\left(p^\alpha m^Z_u\right)s_u\Big),\]
noting that $m^Z_u\ge m_u\ge \bar m_u$ for all $u\in\cS_\alpha$. Thus, it only remains to prove that for any $S\subset\cS_\alpha$ such that $S\supset\{u\in\cS_\alpha:\bar m_u=0\}$, we have
\[\bbP_p\left(\forall u\in S,J^u_{A\cap X}>\e s_u/L\right)\le \exp\Big(-\sum_{u\in S}h_p^u\left(m_u^Zp^\alpha\right)s_u\Big).\]

Fix $u\in S$ such that $s_u$ is maximal. Since $u\in S$, there exist at least $\e s_u/(CKL)$ disjoint $3K$-connected sets of $\alpha$ infections in $X\setminus\bbH_u(a_u)$. But by construction $|X|\le Cs_u Tp^{-\alpha}$, so the union bound gives
\begin{align*}\bbP_p\left(J^u_{A\cap X}>\e s_u/L\right)&{}\le p^{\alpha \e s_u/(CKL)}\binom{K^C s_u Tp^{-\alpha}}{\e s_u/(CKL)}\\
&\le p^{\alpha \e s_u/(CKL)}\left(\frac{eK^C s_u Tp^{-\alpha}}{\e s_u/(CKL)}\right)^{\e s_u/(CKL)}\\
&{}\le \left(K^{2C}L T/\e\right)^{\e s_u/(CKL)}\le \exp(-Ls_u)\\
&\le \exp\Big(-\sum_{v\in \cS_\alpha\setminus V}h^v_p\left(m^Z_vp^{\alpha}\right)s_v\Big),\end{align*}
since $T$ is chosen small enough depending on $\e,C,K,Z,L$ and $h^v_p(m^Z_vp^\alpha)\le h^v_p(Z)<L/|\cS_\alpha|$, since $L$ is chosen large enough depending on $Z$.
\end{proof}

\subsection{Hierarchies}
We next introduce the notion of hierarchies we will use, following \cite{Holroyd03}, where this method was introduced. 

\begin{definition}[Hierarchy]
\label{def:hierarchy}
Let $D$ be a nonempty $\cS_\alpha$-droplet. A \emph{hierarchy} $\HH=(V_\cH,E_\cH)$ for $D$ is an oriented rooted tree with edges pointing away from the root and the following additional structure. Each vertex $v\in V_\cH$ is labelled by a non-empty $\cS_\alpha$-droplet $D_v$. Let $N(v)$ denote the out-neighbourhood of $v$. We require the following conditions to hold.
\begin{enumerate}
    \item The label of the root is $D$.
    \item For any $v\in V_\cH$, $|N(v)|\le 2$.
    \item For any $v\in V_\cH$ and $u\in N(v)$, $D_u\subseteq D_v$.
    \item If $v\in V_\cH$ and $N(v)=\{u,w\}$ with $u\neq w$, then $D_{u}\cup D_w$ is $K$-connected and $D_v=D_u\vee D_w$ (recall \cref{def:span}).
\end{enumerate}
\end{definition}
Vertices of $v\in V_\cH$ are called \emph{seeds}, \emph{normal vertices} and \emph{splitters} if $|N(v)|=0,1,2$ respectively.

\begin{definition}[Precision of  a hierarchy]
\label{def:precision}
Let $z\geq |\cS_\alpha|$ and $t>0$. A hierarchy \emph{of precision} $(t,z)$ is a hierarchy $\cH$ such that the following hold.
\begin{enumerate}
\item A vertex $v\in V_\cH$ is a seed if and only if $\Phi(D_v)\le z$.
\item If $N(u)=\{v\}$, then $\Psi(D_v,D_u) \leq t$.
\item If $v\in N(u)$ and either $u$ is a splitter or $v$ is a normal vertex, then $\Psi(D_v,D_u) >t/2$.
\end{enumerate}
\end{definition}

We now relate the concept of hierarchy to our study.
\begin{definition}[Occurrence of a hierarchy] 
\label{def:occurrence}
A hierarchy \emph{occurs} if the following disjoint occurrence event holds (recall \cref{subsec:proba:tools}):
\[\cE(\cH)=\mathop{\bigcirc}\limits_{\substack{u\in V_\cH,\\N(u)=\varnothing}}\cE(D_u)\circ \mathop{\bigcirc}\limits_{\substack{u,v\in V_\cH,\\N(u)=\{v\}}}\cE(D_v,D_u).\]
\end{definition}
The proof of the following key deterministic result is omitted, as it is identical to \cite{Bollobas23}*{Lemma 8.7}.
\begin{proposition}[Existence of a hierarchy] \label{hierarchy} Let $z\geq |\cS_\alpha|$, $t>0$ and $D$ be a non-empty $\cS_\alpha$-droplet. If $D$ is spanned, then there exists a hierarchy of precision $(t,z)$ for $D$ that occurs.
\end{proposition}

The next lemma allows us to bound the number of hierarchies in order to use the union bound on their occurrence. For the purposes of counting, we identify $\cS_\alpha$-droplets with their intersection with $\bbZ^d$.
\begin{lemma}[Number of hierarchies] \label{number of hierarchies}Fix $a>0$. Let $t>0$ and $z\ge |\cS_\alpha|$. Let $D$ be a $\cS_\alpha$-droplet such that $\Phi(D)/t\le a$. Then, there exists a constant $c(a)>0$ such that the number of hierarchies for $D$ of precision $(t,z)$ is at most $c(a) \Phi(D)^{c(a)}$.
\end{lemma}
\begin{proof} The definition of the hierarchy of precision $(t,z)$ implies that every two steps away from the root, the total location of droplets decreases by at least $t/2$. Therefore, the height of a hierarchy with root label $D=D[\ba]$ is at most  $4\sum_{u\in\cS_\alpha}a_u/t\le C_{0}\Phi(D)/t$ for a suitably large $C_{0}>0$. In particular, there is a bounded number of possible tree structures for $\cH$ (without the labels). Moreover, for each label the number of possibilities is at most $C_{0}\Phi(D)^{|\cS_\alpha|}$, since $C_{0}$ is large enough. Indeed, for each $u\in\cS_\alpha$ the number of $n$ such that $l_u(n)\cap D\neq\varnothing$ is at most of order $\Phi(D)$ and those are the possible choices of $a_u$ in the radii $\ba$ defining the labelling droplet $D[\ba]$.
\end{proof}

\subsection{The probability of occurrence of a hierarchy}
In order to use a union bound on hierarchies, we will need to estimate $\bbP_p(\cE(\cH))$ for a given hierarchy $\cH$. If $\cH$ involves no splitters, this is straightforward, as one can directly apply \cref{prop:bound:E}. Even though this is the dominant scenario, we will need to account for all other possibilities as well. Naturally, the main issue are hierarchies with many splitters and, therefore, many seeds. It is therefore natural to introduce the following quantity, still following \cite{Holroyd03}.
\begin{definition}[Pod of a hierarchy]
The \emph{pod} of a hierarchy $\HH$, denoted by $\Pod(\HH)$, is defined by
\[\Pod(\HH)=\sum_{\substack{u\in V_\cH,\\N(u)=\varnothing}}\Phi(D_u).\]
\end{definition}

Before dealing with an entire hierarchy, we first bound the probability of a single seed. Let us note that a more general statement can be found in \cite{Hartarsky22univlower}*{Corollary A.11}, but in the symmetric setting we are dealing with one has an easier way to achieve the following.
\begin{lemma}[Seed bound]
\label{lem:seeds}
For any $\cS_\alpha$-droplet $D$,
\[\bbP_p(\cE(D))\le \begin{cases}e^{-\Phi(D)L}&\text{if }\Phi(D)\le Z/p^\alpha,\\
e^{-\Phi(D)/L}&\text{if }\Phi(D)\le CB/p^\alpha.
\end{cases}\]
\end{lemma}
\begin{proof}
Let $D=D[\ba]$ for $\ba\in\bbR^{\cS_\alpha}$. Fix $u\in\cS_\alpha$ such that $a_u-a_{-u}=\max_{w\in\cS_\alpha}(a_w-a_{-w})$. Up to translating, we may assume that $D$ is contained in the rectangle $R^u(m,n)$ with $n=a_u-a_{-u}\ge 2\Phi(D)/C$ and $m=\max(Tp^{-\alpha},\Phi(D))$. Finally, observe that the event $\cE(D)$ implies that $\cC^{u}(m,n,\varnothing)$ from \cref{eq:def:CumnE} also occurs. Then, \cref{upper bound half-plane} gives
\begin{equation}
\label{eq:hideous}\bbP_p(\cE(D))\le \exp\Big(-\min_{u\in\cS_\alpha}h^u_p\Big(\max\Big(T,\Phi(D)p^{\alpha}\Big)\Big)\Phi(D)/C\Big).\end{equation}
Since $h^u_p$ is non-increasing, $h^u_p\to h^u$ by \cref{prop:h}\ref{item:3} and $h^u\to \infty$ as $x\to 0$ by \cref{prop:h}\ref{item:2}, we have $h^u_p(\max(T,\Phi(D)p^\alpha))\ge h^u_p(Z)\ge CL$ for all $p$ small enough, all $u\in\cS_\alpha$, if $\Phi(D)\le Zp^{-\alpha}$. Similarly, if $Z/p^\alpha<\Phi(D)\le CB/p^\alpha$, we have $h^u_p(\max(T,\Phi(D)p^\alpha))\ge h^u_p(CB)\ge C/L$. Plugging these bounds into \cref{eq:hideous} concludes the proof.
\end{proof}
Applying the BK inequality \cref{eq:BK} to \cref{lem:seeds}, we immediately obtain the following.
\begin{corollary}\label{seed perimeter}
 Let $\cH$ be a hierarchy for $D$ of precision $(T/p^{\alpha},Z/p^{\alpha})$. Then
\[\P_p\Big(\mathop{\bigcirc}\limits_{\substack{u\in V_\cH,\\N(u)=\varnothing}}\cE(D_u)\Big)\leq \exp (-L\Pod(\cH)).\]
\end{corollary}

If $\Pod(\cH)\ge 2\lambda/(Lp^\alpha)$, \cref{seed perimeter} will be sufficient to conclude. In order to deal with the more relevant hierarchies with smaller pods, we will need a more precise bound. 

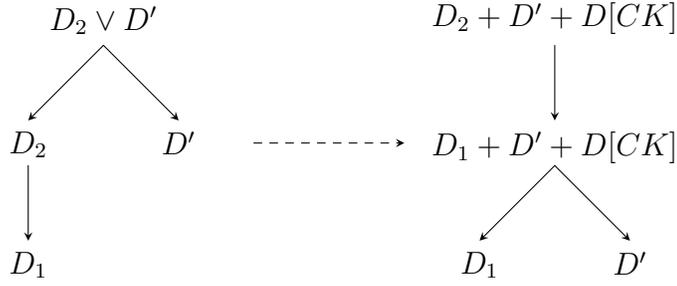
\begin{figure}
\centering
\begin{tikzpicture}[line cap=round,line join=round,>=stealth,x=1cm,y=1cm]
\draw[->] (0,0) node[above]{$D_2\vee D'$}--(-1,-1) node[below]{$D_2$};
\draw[->] (-1,-1.6)--(-1,-2.6) node[below]{$D_1$};
\draw[->] (0,0)--(1,-1) node[below]{$D'$};
\draw[->,dashed] (2,-1.3)--(4,-1.3);
\begin{scope}[shift={(6,0)}]
\draw[->] (0,0) node[above]{$D_2+ D'+ D[CK]$}--(0,-1) node[below]{$D_1+ D'+ D[CK]$};
\draw[->] (0,-1.6)--(-1,-2.6) node[below]{$D_1$};
\draw[->] (0,-1.6)--(1,-2.6) node[below]{$D'$};
\end{scope}
\end{tikzpicture}
\caption{The operation on hierarchies provided by \cref{chain property,dimension}. The seeds are identical, the work functional for the normal vertex on the left is larger than the one on the right, while the root is labelled by a larger droplet on the right, so the second `hierarchy' is more efficient. However, since $D_1$ and $D'$ have no reason to be $K$-connected, the result on the right is no longer a proper hierarchy.}
\label{fig:tree:transformation}
\end{figure}

The goal of the next two lemmas is, roughly speaking, to transform a hierarchy with a splitter root into one with a normal root, as depicted in \cref{fig:tree:transformation}. The first lemma is essentially \cite{Bollobas23}*{Eq. (16)}, so we omit the proof.
\begin{lemma}[Sub-additivity of the span]\label{dimension}
Assume $D_1,D_2,D$ are $\cS_\alpha$-droplets such that $D_1\cup D_2$ is $K$-connected. Then some translate of $D_1+ D_2+ D[CK]$ contains $D_1\vee D_2$.
\end{lemma}
\begin{lemma}\label{chain property}Let $D_1\subseteq D_2$ and $D'$ be three $\cS_\alpha$-droplets. We have 
\[W_p(D_1,D_2)\geq W_p(D_1 + D',D_2 + D').\]
\end{lemma}
\begin{proof}
This follows from the fact that $h^u_p$ is non-decreasing and \cref{loc}.
\end{proof}
\Cref{chain property} is the main reason why the infection forms droplets. It is always more efficient for the infections to appear near existing infected droplets. Hence, the dynamics has a tendency to create large droplets.

As a result of the operation from \cref{fig:tree:transformation} and \cref{prop:bound:E}, we obtain the following bound.
\begin{proposition}
\label{pod of H} 
Let $D$ be a $\cS_\alpha$-droplet with $\Phi(D)\le CBp^{-\alpha}$. For any hierarchy $\HH$ of precision $(Tp^{-\alpha},Zp^{-\alpha})$ for $D$ with $N-1$ normal vertices and $S$ splitters, there exists a non-decreasing sequence of $\cS_\alpha$-droplets $D_1\subseteq \cdots\subseteq D_N$ satisfying
\begin{itemize}
\item $\Phi(D_1)\leq B S+\Pod(\HH)$,
\item either $Bp^{-\alpha}\le \Phi(D_N)\le  CBp^{-\alpha}$, or both $\Phi(D_N)< Bp^{-\alpha}$ and $D_N\supseteq D$,
\item $\bbP_p(\cE(\cH))\le C^{N}\exp(-(1-\e)^2\sum_{i=1}^{N-1}W_p(D_i^{Zp^{-\alpha}},D_{i+1}^{Zp^{-\alpha}})/p^\alpha)$.\end{itemize}
\end{proposition}

\begin{proof}
We proceed by induction on hierarchies. Let $D_r$ be the label of the root of $\HH$.

\noindent\textbf{Case 1.} Assume the root $r$ is a seed. Then $\HH$ is a singleton, $N=1$ and it is sufficient to set $D_1=D_r$.

\noindent\textbf{Case 2.} Assume the root $r$ is a normal vertex. Let $N(r)=\{u\}$. The induction hypothesis for the hierarchy with $r$ removed yields a sequence $D_1\subseteq\dots\subseteq D_{N-1}$ of $\cS_\alpha$-droplets. If $Bp^{-\alpha}\le \Phi(D_{N-1})\le  CBp^{-\alpha}$, we set $D_N=D_{N-1}$ and we are done. 

Assume that, on the contrary, $D_{N-1}\supseteq D_u$ and $\Phi(D_{N-1})< Bp^{-\alpha}$. In this case we set $D_N=D_r\vee D_{N-1}$. The resulting sequence clearly satisfies the first condition. Since $r$ is a normal vertex, by \cref{def:precision} we have $\Psi(D_u,D_r)\le Tp^{-\alpha}$, so $D_N\subseteq D_{N-1}+D[CTp^{-\alpha}]$. We further claim that $\Psi(D_{N-1},D_N)\le \Psi(D_u,D_r)\le Tp^{-\alpha}$. To see this, let $\ba,\bb,\bc,\bd\in\bbR^{\cS_\alpha}$ denote the radii of $D_u,D_r,D_{N-1},D_N$ respectively, so that $\bc\ge \ba$ and $\bd=\bb\vee\bc$. Then indeed 
\[\Psi(D_{N-1},D_N)=\sum_{v\in\cS_\alpha}(d_v-c_v)=\sum_{v\in\cS_\alpha}(\max(0,b_v-c_v))\le \sum_{v\in\cS_\alpha}(\max(0,b_v-a_v))=\Psi(D_u,D_r).\]
Therefore, $\Phi(D_N)\le CBp^{-\alpha}$ and $D_N\supseteq D_{r}$ by \cref{def:span}, so the second condition is also satisfied. Note that $D_N=D_r\vee D_{N-1}$ and $D_{N-1}\supseteq D_u$, so $\cE(D_u,D_r)\subseteq\cE(D_{N-1},D_N)$. Thus, applying \cref{prop:bound:E}, we get
\[\P_p(\E(D_u,D_r))\leq \P_p(\E(D_{N-1},D_N))\leq C\exp\left(-(1-\e)^2W_p\left(D_{N-1}^{Zp^{-\alpha}},D_{N}^{Zp^{-\alpha}}\right)/p^\alpha\right),\]
since $\Psi(D_{N-1},D_N)\le Tp^{-\alpha}$ and $\Phi(D_N)\le CB/p^{-\alpha}$. Combining this with \cref{def:occurrence}, the BK inequality \cref{eq:BK} and the induction hypothesis, we obtain that the third condition of \cref{pod of H} is also fulfilled.

\noindent\textbf{Case 3.} Assume the root $r$ is a splitter. Denote $N(r)=\{u,v\}$ and let $D^u_1,\dots,D^u_{N^u}$ and $D^v_1,\dots,D^v_{N^v}$ be the sequences yielded by the induction hypothesis for the sub-hierarchies $\cH^u,\cH^v$ with roots $u$ and $v$ respectively. Without loss of generality, assume $\Phi(D_{N^u}^u)\ge \Phi(D_{N^v}^v)$. If $\Phi(D^u_{N^u})\ge Bp^{-\alpha}$, then the sequence \[D_i=\begin{cases}D_i^u&i\in\{1,\dots,N^u\},\\
D_{N^u}^u&i\in\{N^u+1,\dots, N^u+N^v-1\}\end{cases}\]
clearly satisfies the desired properties. 

Assume that, on the contrary, $\Phi(D_{N^v}^v)\le \Phi(D^u_{N^u})< Bp^{-\alpha}$. In this case, we seek to implement the transformation of \cref{fig:tree:transformation}. Set
\[D_i=\begin{cases}D[CK]+D^u_1+ D^v_i&i\in\{1,\dots,N^v\},\\
D[CK]+D^u_{i-N^v+1}+ D^v_{N^v}&i\in\{N^v+1,\dots,N^v+N^u-1\}.
\end{cases}\]
Since the perimeter is additive, we have
\[\Phi(D_1)=\Phi(D_1^u)+\Phi(D_1^v)+\Phi(D[CK]),\]so the first condition is met, using the induction hypothesis. We have $D_{N^u+N^v-1}\supseteq D_r$ by \cref{dimension} up to translating the sequence $(D_i)_{i=1}^{N^u+N^v-1}$ appropriately. Moreover, 
\[\Phi(D_{N^u+N^v-1})=\Phi(D_{N^u}^u)+\Phi(D_{N^v}^v)+CK|\cS_\alpha|\le 2Bp^{-\alpha}+B<CBp^{-\alpha},\]
so the second condition is also verified. Finally, the BK inequality and the induction hypothesis give \begin{multline*}\bbP_p(\cE(\cH))\le \bbP_p(\cE(\cH^u))\bbP_p(\cE(\cH^v))\le C^{N^u+N^v}\exp\Bigg(-\frac{(1-\e)^2}{p^\alpha}\\\times\left(\sum_{i=1}^{N^u-1}W_p\left((D_i^u)^{Zp^{-\alpha}},(D_{i+1}^u)^{Zp^{-\alpha}}\right)+\sum_{i=1}^{N^v-1}W_p\left((D_i^v)^{Zp^{-\alpha}},(D_{i+1}^v)^{Zp^{-\alpha}}\right)\right)\Bigg),\end{multline*}
which is enough to conclude, using \cref{chain property}.
\end{proof}

\subsection{Truncating \texorpdfstring{$\lambda_\alpha$}{lambda alpha}}
In order to relate the bound from \cref{pod of H} to the constant $\lambda_\alpha$ from \cref{def:W}, we will need to truncate our bi-infinite sequences of droplets. We start by showing that it is always cheap to extend sequences to $+\infty$.
\begin{lemma}[Extension at $+\infty$]
\label{lem:extention}
For any $\cS_\alpha$-droplet $D$ with $\Phi(D)\ge G$, there exists a sequence of $\cS_\alpha$-droplets $D=D_0\subseteq D_1\subseteq\dots$ such that $\bigcup_{i\ge 0}D_i=\bbR^2$ and $\sum_{i=0}^{\infty}W(D_i,D_{i+1})\le \varepsilon$.
\end{lemma}
\begin{proof}
After translating, we may assume that for some sufficiently large $k$ depending on $\varepsilon$ we have that $D\subseteq D[2^k]$, but $D$ is not contained in any translate of $D[2^{k-1}]$. As we saw in \cref{eq:powers}, taking $k$ large we can ensure that $\sum_{i\ge k}W(D[2^i],D[2^{i+1}])\le \varepsilon/2$. Therefore it suffices to find $D=D_0\subseteq\dots\subseteq D_N=D[2^k]$ such that
$\sum_{i=0}^{N-1}W(D_i,D_{i+1})\le \varepsilon/2$.

In order to achieve this, we proceed similarly to the proof of \cref{lem:sequence}. Set $D=D[\ba^{(0)}]$ and $D[2^k]=D[\ba^{(\infty)}]$. We define $\ba^{(i)}$ by induction as follows, set $D_i=D[\ba^{(i)}]$ and denote by $\bm^{(i)}$ the dimension of $D_i$. Further let $u_i\in\cS_\alpha$ be such that $m^{(i)}_{u_i}=\max\{m^{(i)}_{u}:u\in\cS_{\alpha},a^{(i)}_{u}\neq a^{(\infty)}_u\}$. As long as $D_i\neq D[2^k]$ (at which point the construction is done), we set 
\[\ba^{(i+1)}= \ba^{(i)}+\be_{u_i}\min\left(T,a^{(\infty)}_{u_i}-a^{(i)}_{u_i}\right).\]
This procedure clearly yields $D_N=D[2^k]$ for some finite $N$. Further observe that $m^{(i)}_{u_i}\ge 2^k/C$ for all $i\in\{0,\dots,N-1\}$ and $C>0$ large enough. That is, the largest edge that has not yet reached its final position is always big. Indeed, every two edges of $D_{i}$ that have reached the final value for their radius are necessarily far apart, so there has to be a large side between them. Using this property, we have that 
\[\sum_{i=0}^{N-1}W(D_i,D_{i+1})\le \sum_{u\in\cS_\alpha}h^u\left(2^k/C\right)a^{(\infty)}_{u}\le \varepsilon/2\]
for $k$ large enough, using \cref{prop:h}\ref{item:2}.
\end{proof}
Unfortunately, the analogous statement for extending sequences to $-\infty$ is not true, since arbitrarily small droplets have a divergent cost to produce if they are too elongated. Nevertheless, we are able to obtain the following.

\begin{lemma}[Truncating $\lambda_\alpha$]
\label{lem:truncation}
Let $D_1\subseteq \dots \subseteq D_N$ be a sequence of $\cS_\alpha$-droplets such that $\Phi(D_N)\ge Gp^{-\alpha}$ and $\Phi(D_1)\le 1/(Gp^\alpha)$. Then 
\[\sum_{i=1}^{N-1}W(D_i,D_{i+1})\ge 2\lambda_\alpha-2\varepsilon.\]
\end{lemma}
The gist of the proof is the following. We seek to produce a bi-infinite sequence $\cD'\in\mathfrak D_\alpha$ with $\cW(\cD)\le 2\varepsilon+\sum_{i=1}^{N-1}W(D_i,D_{i+1})$. To do so, we first slightly enlarge all $D_i$, so that the initial droplet becomes roughly circular. The extension of this finite sequence at $+\infty$ is done as in the proof of \cref{lem:extention}. For the extension at $-\infty$, we proceed in two steps. First, we make the droplet exactly circular in a fixed number of steps and then decrease its radii exponentially to infinity.
\begin{proof}
Set $D=D[\Phi(D_1)]$ and set $D'_i=D+D_i$ for all $i\in\{1,\dots,N\}$. By \cref{chain property} we have
\[\sum_{i=1}^{N-1}W(D_i,D_{i+1})\ge \sum_{i=1}^{N-1}W(D'_i,D'_{i+1}).\]
We further use \cref{lem:extention} applied to $D'_N$ to define $D'_i$ for all $i>N$ in such a way that $\sum_{i\ge N}W(D'_i,D'_{i+1})<\varepsilon$. However, now we have ensured that $D'_1$ is roughly circular. Using this fact, up to translation, we can assume that $D[2^{-k-C}]\subseteq D'_1\subseteq D[2^{-k}]$ with $k>0$ large enough depending on $\varepsilon$. We then proceed as in the proof of \cref{lem:extention} to define droplets $D[2^{-k-C}]=D'_{-N'}\subseteq\dots\subseteq D'_1$ for some $N'\ge 0$ in such a way that
\[\sum_{i=-N'}^0W(D'_i,D_{i+1})\le \varepsilon/2.\]
Here, we crucially use that the dimensions of all $D'_i$ for $i\in\{-N',\dots,0\}$ are at least $2^{-k-C}/C$, but the proof is the same as for \cref{lem:extention}. Finally, recalling \cref{eq:powers}, we may set $D'_{i}=D[2^{-k-C+i+N'}]$ for $i<-N'$ to obtain
\[\sum_{i\in\bbZ}W(D'_{i},D_{i+1}')\le 2\varepsilon+\sum_{i=1}^{N-1}W(D_i,D_{i+1}).\]
Since $(D'_i)_{i\in\bbZ}\in\mathfrak D_\alpha$, we are done by the definition of $\lambda_\alpha$.
\end{proof}

\subsection{Proof of the lower bound of Theorem~\ref{th:main}} 
We are ready to upper bound the probability that a droplet of size $Bp^{-\alpha}$ is spanned, using \cref{pod of H}. Once that is done, \cref{th:main} will follow immediately.
\begin{proposition}[Critical spanning bound]\label{upper bound crossed}
For any $\cS_\alpha$-droplet $D$ satisfying $Bp^{-\alpha}\le\Phi(D)\le CBp^{-\alpha}$ we have
\[\P_p(\cE(D))\leq \exp\left(-(2\lambda-C\e)/p^{\alpha}\right).\]
\end{proposition}
\begin{proof}
\Cref{hierarchy} gives that if $\cE(D)$ occurs, then $\cE(\cH)$ does for some hierarchy $\cH$ of precision $(Tp^{-\alpha},Zp^{-\alpha})$ for $D$. Using \cref{number of hierarchies}, we obtain that for some $c(T)>0$ large enough
\[\P_p(\cE(D))\leq c(T)\Phi(D)^{c(T)}\cdot\max_\cH\P_p(\cE(\cH))\leq \exp\left(\e p^{-\alpha}\right)\max_\cH\P_p(\cE(\cH)).\]
It is thus sufficient to prove 
that for any $\cH$
\[\bbP_p(\cE(\cH))\le \exp\left(-(2\lambda-(C-1)\e)p^{-\alpha}\right).\] 

If $\Pod(\HH)\geq 2\lambda/(L p^{\alpha})$, we are done by \cref{seed perimeter}. We therefore assume that $\Pod(\HH)\leq 2\lambda /(Lp^{\alpha})$.
\Cref{pod of H} yields the existence of a sequence $D_1\subseteq\dots\subseteq D_N$ with $\Phi(D_1)<1/(Bp^{\alpha})$ and $\Phi(D_N)\ge Gp^{-\alpha}$ satisfying
\[\P_p(\E(\HH))\leq C^N\exp\left(-(1-\e)^2p^{-\alpha}\sum_{n=1}^{N-1}W_p\left(D_n^{Zp^{-\alpha}},D_{n+1}^{Zp^{-\alpha}}\right)\right).\]

However, by \cref{lem:truncation,prop:lambda:alpha}, we have
\begin{equation}
\label{eq:sum:W}\sum_{n=1}^{N-1} W_p\left(D_n^{Zp^{-\alpha}},D_{n+1}^{Zp^{-\alpha}}\right) \geq 2\lambda-2\e.\end{equation}
Thus, 
\[\P_p(\cE(\cH))\le C^N\exp\left(-(1-2\e)\left(2\lambda-2\e\right)/p^{\alpha}\right).\]
Since $N$ and $C$ do not depend on $p$, this concludes the proof.
\end{proof}
Concluding the proof of \cref{th:main} from \cref{upper bound crossed} is very standard, the argument dating back to \cite{Aizenman88}. We accordingly need the following result from \cite{Bollobas23}*{Lemma 6.18} (see also \cite{Hartarsky22univlower}*{Lemma A.9}).
\begin{lemma}[Aizenman--Lebowitz lemma]
\label{lem:AL}
    Let $D$ be a $\cS_\alpha$-droplet and $1/\varepsilon\le k\le \Phi(D)$. If $\cE(D)$ occurs, then there exists a $\cS_\alpha$-droplet $D'\subseteq D$ with $k/C\le \Phi(D')\le k$ such that $\cE(D')$ occurs.
\end{lemma}
\begin{proof}[Proof of \cref{th:lower}]
Let $\Lambda=\exp (\lambda-C\e)/p^{\alpha}$. Let $\E$ be the event that $0\in[A\cap [-\Lambda,\Lambda]^2]$. We claim that $\bbP_p(\cE)\to0$ as $p\to0$. Indeed, if $\cE$ occurs, then the origin belongs to a spanned $\cS_\alpha$-droplet $D$ with $1\le \Phi(D)\le C\Lambda$. As in \cref{number of hierarchies}, there are at most $(C\Phi(D))^{|\cS_\alpha|}$ possible choices for this $\cS_\alpha$-droplet, given its perimeter. 

\noindent\textbf{Case 1.} Assume that $\Phi(D)\le \log(1/p)$. Then $\bbP_p(\cE(D))\le p|D\cap\bbZ^2|$, so
\begin{equation}
\label{eq:wrap:up:1}\sum_{\substack{D\ni 0,\\\Phi(D)\le \log(1/p)}}\bbP_p(\cE(D))\le p(\log(1/p))^C.
\end{equation}

\noindent\textbf{Case 2.} Assume that  $\log(1/p)\le \Phi(D)\le CB/p^\alpha$. Then by \cref{lem:seeds}, 
\begin{equation}
\label{eq:wrap:up:2}\sum_{\substack{D\ni 0,\\\log(1/p)\le \Phi(D)\le CB/p^\alpha}}\bbP_p(\cE(D))\le C\sum_{\phi\ge \log(1/p)}\phi^Ce^{-\phi/L}\le 1/\log(1/p).
\end{equation}

\noindent\textbf{Case 3.} Assume that  $CB/p^\alpha\le \Phi(D)\le C\Lambda$. Then by \cref{lem:AL}, if $\cE(D)$ occurs, we can find a $\cS_\alpha$-droplet $D'\subset [-C\Lambda,C\Lambda]^2$ with $B/p^\alpha\Phi(D')\le CB/p^\alpha$ such that $\cE(D')$ occurs. Thus, a union bound on $D'$ gives
\begin{equation}
\label{eq:wrap:up:3}\bbP_p\Big(\bigcup_{\substack{D\ni0,\\CB/p^\alpha\le\Phi(D)\le C\Lambda}}\cE(D)\Big)\le p^{-C}\Lambda^2\max_{D'}\bbP_p(\cE(D'))\le p^{-C}e^{-C\varepsilon/p^\alpha},\end{equation}
using \cref{upper bound crossed} for the last inequality.

Summing \cref{eq:wrap:up:1,eq:wrap:up:2,eq:wrap:up:3}, we obtain that $\bbP_p(\cE)\to0$ as $p\to0$, as desired. This concludes the proof of 
\cref{th:lower}.
\end{proof}
\begin{proof}[Proof of \cref{th:main}]
    The result follows directly from \cref{th:upper,th:lower}.
\end{proof}

\section{Future directions}
\label{sec:future}
In conclusion, let us comment on possible generalisations of \cref{th:main}. For the purposes of this discussion, let us recall a few notions from \cite{Bollobas23}. A two-dimensional update family $\cU$ is \emph{critical}, if there exists a semi-circle of $S^1$ containing a finite number of stable directions, but every open semi-circle contains at least one stable direction. The \emph{difficulty} of a critical update family is given by 
\[\alpha=\min_{u\in S^1}\max_{v\in S^1: \<u,v\>>0}\alpha(v),\]
where $\alpha(v)$ is given by \cref{eq:def:alpha}, if $v$ is isolated stable or unstable, and $\alpha(v)=\infty$ otherwise (this definition coincides with \cref{eq:def:alpha:global} for isotropic models). A critical update family of difficulty $\alpha$ is \emph{unbalanced}, if there exist two opposite directions $v,-v\in S^1$ such that $\min(\alpha(v),\alpha(-v))>\alpha$, and \emph{balanced} otherwise.

Firstly, we expect that an adapted version of \cref{th:main} holds for all critical models.
\begin{conjecture}
\label{conj:sharp:threshold}
    Let $\cU$ be a critical two-dimensional update family $\cU$. Then there exists a constant $\lambda=\lambda(\cU)$ such that for all $\varepsilon>0$,
    \[\lim_{p\to 0}\bbP_p\left(\left|\frac{p^\alpha\log\tau}{(\log(1/p))^\gamma}-\lambda\right|>\varepsilon\right)=0,\]
    where $\gamma=0$, if $\cU$ is balanced, and $\gamma=2$, if $\cU$ is unbalanced.
\end{conjecture}
However, in general we do not expect the constant $\lambda$ in \cref{conj:sharp:threshold} to be given by \cref{def:W}. We do believe that \cref{conj:sharp:threshold} holds with $\lambda$ from \cref{def:W} for all isotropic voracious update families. That is, we expect that \cref{th:lower} holds without the symmetry assumption. A natural subsequent step would be to extend our methods to (non-symmetric) balanced voracious models.

\begin{figure}
    \centering
    \begin{tikzpicture}[x=0.5cm,y=0.5cm]
\draw (-2,2) circle (0.3);
\draw (-4,2) circle (0.3);

\draw (3,2) circle (0.3);
\draw (3,0) circle (0.3);
 \draw[step=1,gray,very thin](-5,-1)grid(-1,3);
\draw[step=1,gray,very thin](0,-1)grid(4,3);
\draw (-3,1) node[cross=0.15cm,rotate=0] {};
\draw (2,1) node[cross=0.15cm,rotate=0] {};
\draw (-3,-1.5) node[below]{$U_1$};
\draw (2,-1.5) node[below]{$U_2$};

\begin{scope}[shift={(10,5)}]
\draw (-4,-5) circle (0.3);
\draw (-4,-3) circle (0.3);

\draw (1,-5) circle (0.3);
\draw (3,-5) circle (0.3);
\draw[step=1,gray,very thin](-5,-6)grid(-1,-2);
\draw[step=1,gray,very thin](0,-6)grid(4,-2);
\draw (-3,-4) node[cross=0.15cm,rotate=0] {};
\draw (2,-4) node[cross=0.15cm,rotate=0] {};
\draw (-3,-6.5) node[below]{$U_3$};
\draw (2,-6.5) node[below]{$U_4$};
\end{scope}

\begin{scope}[shift={(20,5)}]
\draw (-5,-6) circle (0.3);
\draw (-3,-6) circle (0.3);
\draw (-1,-6) circle (0.3);
\draw (-5,-2) circle (0.3);
\draw (-3,-2) circle (0.3);
\draw (-1,-2) circle (0.3);
\draw (-5,-4) circle (0.3);
\draw (-1,-4) circle (0.3);
\draw[step=1,gray,very thin](-5,-6)grid(-1,-2);
\draw (-3,-4) node[cross=0.15cm,rotate=0] {};
\draw (-3,-6.5) node[below]{$U_5$};
\end{scope}
\begin{scope}[shift={(25,5)}]
\draw (-5,-6) circle (0.3);
\draw (-4,-6) circle (0.3);
\draw (-3,-6) circle (0.3);
\draw (-2,-6) circle (0.3);
\draw (-1,-6) circle (0.3);
\draw (-5,-2) circle (0.3);
\draw (-4,-2) circle (0.3);
\draw (-3,-2) circle (0.3);
\draw (-2,-2) circle (0.3);
\draw (-1,-2) circle (0.3);
\draw (-5,-4) circle (0.3);
\draw (-1,-4) circle (0.3);
\draw (-5,-3) circle (0.3);
\draw (-1,-3) circle (0.3);
\draw (-5,-5) circle (0.3);
\draw (-1,-5) circle (0.3);
\draw[step=1,gray,very thin](-5,-6)grid(-1,-2);
\draw (-3,-4) node[cross=0.15cm,rotate=0] {};
\draw (-3,-6.5) node[below]{$U_6$};
\end{scope}
\end{tikzpicture}
    \caption{Six update rules with the origin marked by a cross in each.}
    \label{fig:voracity:examples}
\end{figure}
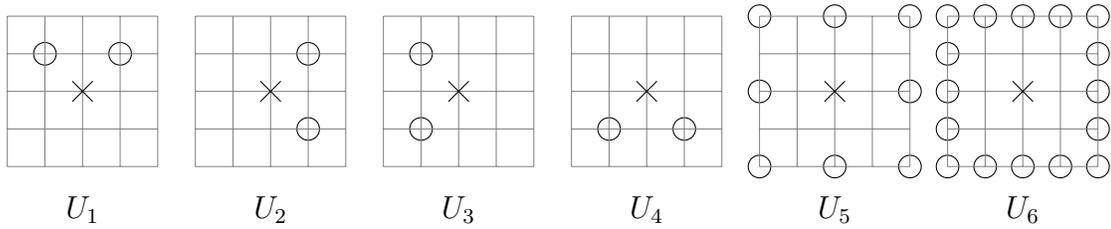

On the other hand, for non-voracious models, \cref{th:upper} does not apply and \cref{th:lower} is not sharp. To provide an example of this, consider the update rules in \cref{fig:voracity:examples} and the update families given by $\cU_1=\{U_1,U_2,U_3,U_4\}$, $\cU_2=\cU_1\cup\{U_5\}$, $\cU_3=\cU_1\cup\{U_6\}$. Viewed as models on $\bbZ^2$, all of three families are isotropic non-voracious and have difficulty $1$. However, $\cU_1$ and $\cU_2$ may be viewed as models on the even sublattice of $\bbZ^2$, in which case they become voracious. We may then apply \cref{th:main} and, inspecting the proof (also see \cite{Holroyd03}), shows that \cref{conj:sharp:threshold} holds with $\lambda=\pi^2/6$ for both $\cU_1$ and $\cU_2$. Moreover, it is not hard to see that $\tau$ for $\cU_3$ is stochastically bounded from above and from below by $\tau$ for $\cU_1$ and $\cU_2$ respectively. Therefore, \cref{conj:sharp:threshold} holds for all three models with $\lambda=\pi^2/6$. Yet, $\cU_3$ cannot be viewed as acting only on the even sublattice. For each of the three models, viewed as acting on all of $\bbZ^2$, the constant $\lambda$ in \cref{def:W} is rather given by $\pi^2/9$. This can be seen along the lines of \cite{Holroyd03}, since helping sets are single sites on the first or second line perpendicular to the stable direction $u=(\pm1/\sqrt 2,\pm1/\sqrt 2)$.

The example above is only the first symptom of the problems arising in the absence of voracity. Indeed, sublattices may have rather complex interactions and, in some cases, it may be more efficient to grow only periodically filled droplets. Furthermore, whether droplets prefer to be completely or partially filled may also depend on the direction. In view of the above, we expect voracity to be the most challenging hypothesis to remove.

\section*{Acknowledgements}
This project was supported by the Austria Science Fund (FWF): P35428-N. This work has received funding from the European Research Council (ERC) under the European Union’s Horizon 2020 research and innovation programme (grant agreements No. 757296). The first author acknowledges funding from the NCCR SwissMap, the Swiss FNS, and the Simons collaboration on localization of waves. For open access purposes, the authors have applied a CC BY public copyright license to any author-accepted manuscript version arising from this submission. We thank Lyuben Lichev and the anonymous referees for helpful comments on the presentation. We also thank Ander Holroyd for suggesting the topic back in 2008 and for helpful discussions.

\bibliographystyle{plain}
\bibliography{Bib}

\begin{bibdiv}
\begin{biblist}

\bib{Aizenman88}{article}{
      author={Aizenman, M.},
      author={Lebowitz, J.~L.},
       title={Metastability effects in bootstrap percolation},
        date={1988},
        ISSN={0305-4470, 1361-6447},
     journal={J. Phys. A},
      volume={21},
      number={19},
       pages={3801\ndash 3813},
         url={https://iopscience.iop.org/article/10.1088/0305-4470/21/19/017},
      review={\MR{968311}},
}

\bib{Balister16}{article}{
      author={Balister, P.},
      author={Bollob{\'a}s, B.},
      author={Przykucki, M.},
      author={Smith, P.},
       title={Subcritical {$\mathcal {U}$}-bootstrap percolation models have
  non-trivial phase transitions},
        date={2016},
        ISSN={0002-9947, 1088-6850},
     journal={Trans. Amer. Math. Soc.},
      volume={368},
      number={10},
       pages={7385\ndash 7411},
         url={http://www.ams.org/tran/2016-368-10/S0002-9947-2016-06586-1/},
      review={\MR{3471095}},
}

\bib{Bohman99a}{article}{
      author={Bohman, T.},
       title={Discrete threshold growth dynamics are omnivorous for box
  neighborhoods},
        date={1999},
        ISSN={0002-9947},
     journal={Trans. Amer. Math. Soc.},
      volume={351},
      number={3},
       pages={947\ndash 983},
         url={https://doi.org/10.1090/S0002-9947-99-02018-8},
      review={\MR{1443863}},
}

\bib{Bollobas17}{article}{
      author={Bollob{\'a}s, B.},
      author={{Duminil-Copin}, H.},
      author={Morris, R.},
      author={Smith, P.},
       title={The sharp threshold for the {{Duarte}} model},
        date={2017},
        ISSN={0091-1798},
     journal={Ann. Probab.},
      volume={45},
      number={6B},
       pages={4222\ndash 4272},
         url={https://doi.org/10.1214/16-AOP1163},
      review={\MR{3737910}},
}

\bib{Bollobas23}{article}{
      author={Bollob{\'a}s, B.},
      author={{Duminil-Copin}, H.},
      author={Morris, R.},
      author={Smith, P.},
       title={Universality for two-dimensional critical cellular automata},
        date={2023},
        ISSN={0024-6115},
     journal={Proc. Lond. Math. Soc. (3)},
      volume={126},
      number={2},
       pages={620\ndash 703},
         url={https://doi.org/10.1112/plms.12497},
      review={\MR{4550150}},
}

\bib{Bollobas15}{article}{
      author={Bollob{\'a}s, B.},
      author={Smith, P.},
      author={Uzzell, A.},
       title={Monotone cellular automata in a random environment},
        date={2015},
        ISSN={0963-5483},
     journal={Combin. Probab. Comput.},
      volume={24},
      number={4},
       pages={687\ndash 722},
         url={http://dx.doi.org/10.1017/S0963548315000012},
      review={\MR{3350030}},
}

\bib{Bringmann12}{article}{
      author={Bringmann, K.},
      author={Mahlburg, K.},
       title={Improved bounds on metastability thresholds and probabilities for
  generalized bootstrap percolation},
        date={2012},
        ISSN={0002-9947},
     journal={Trans. Amer. Math. Soc.},
      volume={364},
      number={7},
       pages={3829\ndash 3859},
         url={https://doi.org/10.1090/S0002-9947-2012-05610-8},
      review={\MR{2901236}},
}

\bib{Chalupa79}{article}{
      author={Chalupa, J.},
      author={Leath, P.~L.},
      author={Reich, G.~R.},
       title={Bootstrap percolation on a {{Bethe}} lattice},
        date={1979},
        ISSN={0022-3719},
     journal={J. Phys. C},
      volume={12},
      number={1},
       pages={L31\ndash L35},
         url={https://iopscience.iop.org/article/10.1088/0022-3719/12/1/008},
}

\bib{Duminil-Copin13}{article}{
      author={{Duminil-Copin}, H.},
      author={{van Enter}, A. C.~D.},
       title={Sharp metastability threshold for an anisotropic bootstrap
  percolation model},
        date={2013},
        ISSN={0091-1798},
     journal={Ann. Probab.},
      volume={41},
      number={3A},
       pages={1218\ndash 1242},
         url={http://dx.doi.org/10.1214/11-AOP722},
      review={\MR{3098677}},
}

\bib{Gravner99}{incollection}{
      author={Gravner, J.},
      author={Griffeath, D.},
       title={Scaling laws for a class of critical cellular automaton growth
  rules},
        date={1999},
   booktitle={Random walks ({{Budapest}}, 1998)},
      series={Bolyai {{Soc}}. {{Math}}. {{Stud}}.},
      volume={9},
   publisher={J{\'a}nos Bolyai Math. Soc., Budapest},
       pages={167\ndash 186},
      review={\MR{1752894}},
}

\bib{Grimmett99}{book}{
      author={Grimmett, G.},
       title={Percolation},
     edition={Second edition},
      series={Grundlehren der mathematischen {{Wissenschaften}}},
   publisher={Springer, Berlin, Heidelberg},
        date={1999},
        ISBN={978-3-540-64902-1},
         url={http://link.springer.com/10.1007/978-3-662-03981-6},
        note={Originally published by Springer, New York (1989)},
      review={\MR{1707339}},
}

\bib{Harris60}{article}{
      author={Harris, T.~E.},
       title={A lower bound for the critical probability in a certain
  percolation process},
        date={1960},
     journal={Math. Proc. Camb. Phil. Soc.},
      volume={56},
      number={1},
       pages={13\ndash 20},
         url={https://doi.org/10.1017/S0305004100034241},
      review={\MR{115221}},
}

\bib{Hartarsky24univupper}{article}{
      author={Hartarsky, I.},
       title={Refined universality for critical {{KCM}}: upper bounds},
        date={2024},
        ISSN={0010-3616},
     journal={Comm. Math. Phys.},
      volume={405},
      number={1},
         url={https://doi.org/10.1007/s00220-023-04874-8},
      review={\MR{4694420}},
}

\bib{Hartarsky22univlower}{article}{
      author={Hartarsky, I.},
      author={Mar{\^e}ch{\'e}, L.},
       title={Refined universality for critical {{KCM}}: lower bounds},
        date={2022},
        ISSN={0963-5483},
     journal={Combin. Probab. Comput.},
      volume={31},
      number={5},
       pages={879\ndash 906},
      review={\MR{4472293}},
}

\bib{Hartarsky23FA}{article}{
      author={Hartarsky, I.},
      author={Martinelli, F.},
      author={Toninelli, C.},
       title={Sharp threshold for the {{FA-2f}} kinetically constrained model},
        date={2023},
        ISSN={0178-8051},
     journal={Probab. Theory Related Fields},
      volume={185},
      number={3},
       pages={993\ndash 1037},
      review={\MR{4556287}},
}

\bib{Hartarsky20a}{article}{
      author={Hartarsky, I.},
      author={Mezei, T.~R.},
       title={Complexity of two-dimensional bootstrap percolation difficulty:
  algorithm and {{NP-hardness}}},
        date={2020},
        ISSN={0895-4801, 1095-7146},
     journal={SIAM J. Discrete Math.},
      volume={34},
      number={2},
       pages={1444\ndash 1459},
         url={https://epubs.siam.org/doi/10.1137/19M1239933},
      review={\MR{4117299}},
}

\bib{Hartarsky24locality}{article}{
      author={Hartarsky, I.},
      author={Teixeira, A.},
       title={Bootstrap percolation is local},
        date={2024},
     journal={arXiv e-prints},
      eprint={2404.07903},
         url={http://arxiv.org/abs/2404.07903},
}

\bib{Holroyd03}{article}{
      author={Holroyd, A.~E.},
       title={Sharp metastability threshold for two-dimensional bootstrap
  percolation},
        date={2003},
        ISSN={0178-8051},
     journal={Probab. Theory Related Fields},
      volume={125},
      number={2},
       pages={195\ndash 224},
         url={http://dx.doi.org/10.1007/s00440-002-0239-x},
      review={\MR{1961342}},
}

\bib{Holroyd03a}{article}{
      author={Holroyd, A.~E.},
      author={Liggett, T.~M.},
      author={Romik, D.},
       title={Integrals, partitions, and cellular automata},
        date={2003},
        ISSN={0002-9947, 1088-6850},
     journal={Trans. Amer. Math. Soc.},
      volume={356},
      number={8},
       pages={3349\ndash 3368},
         url={http://www.ams.org/tran/2004-356-08/S0002-9947-03-03417-2/},
      review={\MR{2052953}},
}

\bib{Kogut81}{article}{
      author={Kogut, P.~M.},
      author={Leath, P.~L.},
       title={Bootstrap percolation transitions on real lattices},
        date={1981},
        ISSN={0022-3719},
     journal={J. Phys. C},
      volume={14},
      number={22},
       pages={3187\ndash 3194},
         url={https://iopscience.iop.org/article/10.1088/0022-3719/14/22/013},
}

\bib{BK85}{article}{
      author={{van den Berg}, J.},
      author={Kesten, H.},
       title={Inequalities with applications to percolation and reliability},
        date={1985},
        ISSN={0021-9002},
     journal={J. Appl. Probab.},
      volume={22},
      number={3},
       pages={556\ndash 569},
      review={\MR{799280}},
}

\bib{vanEnter87}{article}{
      author={{van Enter}, A. C.~D.},
       title={Proof of {{Straley}}'s argument for bootstrap percolation},
        date={1987},
        ISSN={0022-4715},
     journal={J. Stat. Phys.},
      volume={48},
      number={3-4},
       pages={943\ndash 945},
         url={http://dx.doi.org/10.1007/BF01019705},
      review={\MR{914911}},
}

\end{biblist}
\end{bibdiv}
\end{document}